\documentclass[a4paper,12pt,reqno]{amsart}
\usepackage{amssymb, amsmath, amsfonts, amscd}
\usepackage{mathrsfs, latexsym}
\usepackage[all,ps,cmtip]{xy}
\usepackage{array, longtable}
\usepackage{bm, enumitem}
\usepackage{xcolor}
\title{Construction of minimal varieties from quasi-smooth weighted complete intersections}
\date{\today}

\author{Pinxian Bie }
\address{\rm School of Mathematical Sciences, Fudan University, Shanghai 200433, China}
\email{23110180002@fudan.edu.cn}

\newcommand{\bC}{{\mathbb C}}
\newcommand{\bQ}{{\mathbb Q}}
\newcommand{\bP}{{\mathbb P}}

\newcommand{\rounddown}[1]{\lfloor{#1}\rfloor}

\newcommand\Vol{\text{\rm Vol}}

\newcommand\OO{{\mathcal{O}}}

\newcommand{\modr}{\text{\rm mod}}

\newtheorem{thm}{Theorem}[section]
\newtheorem{lem}[thm]{Lemma}

\newtheorem{prop}[thm]{Proposition}

\theoremstyle{definition}
\newtheorem{defn}[thm]{Definition}

\newtheorem{exmp}[thm]{Example}

\newtheorem{rem}[thm]{Remark}

\theoremstyle{remark}

\begin{document}
\begin{abstract} 
 This paper is devoted to the generalization of the construction of minimal varieties from the previous work of Meng Chen, Chen Jiang and Binru Li. We first establish several effective nefness criterions for the canonical divisor of weighted blow-ups over a weighted complete intersection, we both consider the high codimensional case and the blowing up several points case, from which we construct plenty of new minimal $3$-folds including $79$ families of minimal $3$-folds of general type, several infinite series of minimal $3$-folds of Kodaira dimension $2$, $16$ families of minimal $3$-folds of general type on or near the Noether lines. 
\end{abstract}
\maketitle

%%%%%%%%%
\pagestyle{myheadings}
\markboth{\hfill P. Bie\hfill}{\hfill Construction of Minimal varieties from weighted complete intersections\hfill}
\numberwithin{equation}{section}
%%%%%%%%%%%%
%\tableofcontents

\section{Introduction}

In the study of birational classification of algebraic varieties, the Minimal Model Program (MMP) stands as one of the most profound achievements of modern algebraic geometry. This framework establishes that every projective variety admits a birational model that is either a minimal variety (characterized by nef canonical divisor and $\mathbb{Q}$-factorial terminal singularities) or a Mori fiber space. While foundational results like \cite{KMM, K-M, BCHM} have solidified the theoretical foundations, significant challenges remain - most notably the Abundance conjecture.

Beyond these theoretical considerations, an equally vital aspect of classification theory involves the explicit construction of minimal models and the precise computation of their birational invariants. Such concrete realizations serve as crucial testing grounds for theoretical predictions or supporting examples and often reveal unexpected phenomena in higher dimensions. The historical trajectory of this investigative approach can be traced from Reid's classification of weighted K3 surfaces \cite{Reid79} through Iano-Fletcher's systematic enumeration of threefold examples \cite{Fle00}.

The ongoing pursuit of constructing and analyzing minimal varieties - evidenced by works including \cite{JK01, CR02, Rei05, BKR, BR1, B-K, BR2} and the more recent works \cite{TW, ETW, HZ25} - continues to yield both theoretical insights and computational challenges. Recent advances in this direction have further illuminated the delicate interplay between singularity theory, pluricanonical systems, and birational invariants.

Explicit construction of minimal varieties in dimension at least $3$ faces several difficulties. A naive way of constructing desired varieties could be, starting from a singular complete intersection in $\bP^n$, to obtain a smooth model by resolving singularities and to calculate their birational invariants. Unlike the surface case, this would be very difficult in higher dimensions as MMP and birational morphisms are more complicated than those in the surface case. Another difficult part is to determine the nefness condition of the canonical divisor, which requires some subtle examinations.

In this paper, we improve and generalize the privous work of Meng Chen, Chen Jiang and Binru Li in \cite{CJL24}, where they considered a special construction starting from a well-formed quasismooth weighted hypersurface $X$ with only isolated singularities, which has exactly one non-canonical singularity $P$. One takes a weighted blow-up along the center $P\in X$ to obtain a higher model $Y$. 
To make sure that $Y$ is almost minimal, they proved an effective nefness criterion:

\begin{thm}[{cf. \cite[Theorem 1.3]{CJL24}}]\label{nefness 1}
Let $X=X^n_d\subset \mathbb{P}(b_1, \dots, b_{n+2})$ be an $n$-dimensional well-formed quasismooth general hypersurface of degree $d$ with $\alpha=d-\sum_{i=1}^{n+2}b_i>0$ where $b_1, \dots, b_{n+2}$ are not necessarily sorted by size. Denote by $x_1,\dots,x_{n+2}$ the homogenous coordinates of $\mathbb{P}(b_1, \dots, b_{n+2})$.
Denote by $\ell$ the line $(x_1=x_2=\dots =x_{n}=0)$ in $\mathbb{P}(b_1, \dots, b_{n+2})$. Suppose that $X\cap\ell$ consists of finitely many points and take $Q\in X\cap \ell$.
Assume that $X$ has a cyclic quotient singularity of type $\frac{1}{r}(e_1,\dots, e_n)$ at $Q$ where $e_1, \dots, e_n>0$, $\gcd(e_1, \dots, e_n)=1$, $\sum_{i=1}^ne_i<r$ and that $x_1,\dots,x_n$ are also the local coordinates of $Q$ corresponding to the weights $\frac{e_1}{r},\dots,\frac{e_n}{r}$ respectively. Let $\pi: Y\to X$ be the weighted blow-up at $Q$ with weight $(e_1,\dots, e_n)$.
Suppose that there exists an index $k\in \{1,\dots,n\}$ such that 
\begin{enumerate}
\item $\alpha e_{j} \geq b_{j}(r-\sum_{i=1}^ne_i)$ for each $j\in \{1, \dots, \hat{k},\dots, n\}$;
\item $\alpha dre_{k}\geq b_{k}b_{n+1}b_{n+2}(r-\sum_{i=1}^ne_i);$
\item a general hypersurface of degree $d$ in $\mathbb{P}(b_{k}, b_{n+1}, b_{n+2})$ is irreducible;
\item $\mathbb{P}(e_1,\dots,e_n)$ is well-formed.
\end{enumerate}
Then $K_Y$ is nef and $\nu(Y)\geq n-1$ where $\nu(Y)$ denotes the numerical Kodaira dimension. \end{thm}

In practice they also require the resulting $Y$ has at worst canonical singularities, so $Y$ will correspond to a desired minimal variety of Kodaira dimension at least $n-1$, which is the case at least in dimension $3$. It is natural to ask whether one can generalize the setting of Theorem~\ref{nefness 1}, so that we could expect to construct some new examples of minimal varieties. In fact the generalization can be done in several directions. First, a natural generalization is to generalize the previous theorem to the case of weighted complete intersections, which we will prove a nefness criterion similar to Theorem~\ref{nefness 1}. Second, it's also natural to ask whether we can blow up more than one point but still preserve the nefness criteria, this however requires the blown up points have some special positions, and some partial results have already been discussed in \cite{CJL24}. Finally, in some circumstances when we perform a weighted blow up centered at a non-canonical cyclic quotient singularity point, there may still exist another non-canonical singularity on the exceptional divisor, so one may ask when a sequence of weighted blow ups can preserve the nefness of canonical divisors while making the resulting variety has only canonical singularity. We state our theorems as below:

\begin{thm}[Nefness criterion I]\label{nefness 2}
Let $X=X^n_{d_1,d_2,...d_c}\subset \mathbb{P}(b_1, \dots, b_{n+c+1})$ be an $n$-dimensional well-formed quasismooth general complete intersection of multi-degrees $(d_1,...d_c)$ with $\alpha=\sum_{l=1}^{c} d_l-\sum_{j=1}^{n+2}b_j>0$ where $b_1, \dots, b_{n+c+1}$ are not necessarily sorted by size. Denote by $x_1,\dots,x_{n+c+1}$ the homogeneous coordinates of $\mathbb{P}(b_1, \dots, b_{n+c+1})$.
Denote by $\Pi$ the $c$-dimensional linear subspace $(x_1=x_2=\dots =x_{n}=0)$ in $\mathbb{P}(b_1, \dots, b_{n+c+1})$. Suppose that $X\cap\Pi$ consists of finitely many points and take $Q\in X\cap \Pi$.
Assume that $X$ has a cyclic quotient singularity of type $\frac{1}{r}(e_1,\dots, e_n)$ at $Q$ where $e_1, \dots, e_n>0$, $\gcd(e_1, \dots, e_n)=1$, $\sum_{i=1}^ne_i<r$ and that $x_1,\dots,x_n$ are also the local coordinates of $Q$ corresponding to the weights $\frac{e_1}{r},\dots,\frac{e_n}{r}$ respectively. Let $\pi: Y\to X$ be the weighted blow-up at $Q$ with weight $(e_1,\dots, e_n)$.
Suppose that either there exists an index $k\in \{1,\dots,n\}$ such that 
\begin{enumerate}
\item $\alpha e_{j} \geq b_{j}(r-\sum_{i=1}^ne_i)$ for each $j\in \{1, \dots, \hat{k},\dots, n\}$;
\item $\alpha (\prod_{l=1}^cd_l)re_{k}\geq b_{k}(\prod_{j=n+1}^{n+c+1}b_{j})(r-\sum_{i=1}^ne_i);$
\item a general weighted complete intersection curve of multi-degrees $(d_1,..d_c)$ in $\mathbb{P}(b_{k}, b_{n+1},... b_{n+c+1})$ is irreducible;
\item $\mathbb{P}(e_1,\dots,e_n)$ is well-formed.
\end{enumerate}
Or we have
\begin{enumerate}
    \item $\alpha e_{j} \geq b_{j}(r-\sum_{i=1}^ne_i)$ for each $j\in \{1, \dots, n\}$;
    \item $Z_{d_1,d_2,...d_c} \subset \mathbb{P}( b_{n+1},... b_{n+c+1})$ consists of finite points;
    \item $\mathbb{P}(e_1,\dots,e_n)$ is well-formed.
\end{enumerate}
Then $K_Y$ is nef and $\nu(Y)\geq n-1$ where $\nu(Y)$ denotes the numerical Kodaira dimension. \end{thm}

\begin{thm}[Nefness criterion II]\label{nefness 3}
Let $X^{n}_{d_1,...d_c}\subset \mathbb{P}(b_1, \dots, b_{n+c+1})$ be an $n$-dimensional well-formed quasismooth general hypersurface of multi-degrees $(d_1,...d_c)$ with $\alpha=\sum_{j+1}^{c}d_j-\sum_{i=1}^{n+c+1}b_i>0$, where $b_1,\dots, b_{n+c+1}$ are not necessarily ordered by size. Denote by $x_1,\dots,x_{n+c+1}$ the homogenous coordinates of $\mathbb{P}(b_1, \dots, b_{n+c+1})$. There are two particular cases that we can handle:

\textbf{Case 1}
\begin{itemize}
\item Denote by $\Pi_t$ be the  corresponding linear subspaces $(x_1=x_2=\dots =x_{n-1}=x_{n+t-1}=0)$ in $\mathbb{P}(b_1, \dots, b_{n+c+1})$ for all $t\in \{1,2,...s\}$ respectively, where $s\leq c+2$. Suppose that for all $t\in \{1,2,...s\}$, $X\cap \Pi_t$ consist of finitely many points and take $Q_t\in X\cap \Pi_t$.

\item Assume that $X$ has $s$ cyclic quotient singularities at $Q_t$ of type $\frac{1}{r_t}(e_{1,t},\dots,e_{n-1,t}, e_{n,t})$ for all $t\in \{1,2,...s\}$ where $e_{i,t}$ are positive integers, $\gcd(e_{1,t}, \dots, e_{n,t})=1$, $\sum_{i=1}^ne_{i,t}<r_t$. 

\item Assume further that $x_1,\dots, x_{n-1}, x_{n+t-1}$ are the local coordinates of $Q_t$ corresponding to the weights $\frac{e_{1,t}}{r_t},\dots,\frac{e_{n,t}}{r_t}$ for all $t\in \{1,2,...s\}$. 
\end{itemize}

Take $\pi: Y\to X$ to be the weighted blow-up at those $Q_t$ with  corresponding weights $(e_{1,t},\dots,e_{n,t})$.
Suppose that the following conditions hold: 
\begin{enumerate}
\item $\alpha e_{j,t} \geq b_{j}(r_t-\sum_{i=1}^ne_{i,t})$ for each $j\in \{1, \dots, n-1\}$ and for each $t\in \{1,2,...s\}$;
\item $\frac{\alpha \prod_{k=1}^{c}d_k}{\prod_{j=n}^{n+c+1}b_j}\geq \sum_{t=1}^{s} \frac{r_t-\sum_{i=1}^ne_{i,t}}{r_te_{n,t}}$ ;
\item a general weighted complete intersection curve of multi-degrees $(d_1,...d_c)$ in $\mathbb{P}(b_{n}, ..., b_{n+c+1})$ is irreducible;
\item $\mathbb{P}(e_{1,t},\dots,e_{n,t})$  are well-formed for each $t\in \{1,2,...s\}$.
\end{enumerate}

\textbf{Case 2}
\begin{itemize}
\item Denote by $\Pi$ be the  corresponding linear subspace $(x_1=x_2=\dots =x_{n}=0)$ in $\mathbb{P}(b_1, \dots, b_{n+c+1})$. Suppose that $X\cap \Pi$ consists of finite points and take $Q_u \in X\cap \Pi$ for $u\in \{1,2,....N\}$.
\item Assume $X$ has a cyclic quotient singularity at each $Q_u$ of type $\frac{1}{r_u}(e_{1,u},\dots,e_{n-1,u}, e_{n,u})$ where $e_{i,u}$ are positive integers, $\gcd(e_{1,u}, \dots, e_{n,u})=1$, $\sum_{i=1}^ne_{i,u}<r_u$.
\item Assume further that $x_1,\dots, x_{n}$ are the local coordinates of $Q_u$ corresponding to their respective weights.
\end{itemize}

Take $\pi: Y\to X$ to be the weighted blow-up at those $Q_u$ with  corresponding weights $(e_{1,u},\dots,e_{n,u})$.
Suppose there exists a index $k\in \{1,\dots,n\}$ such that : 
\begin{enumerate}
\item $\alpha e_{j,u} \geq b_{j}(r_u-\sum_{i=1}^ne_{i,u})$ for each $j\in \{1, \dots, \hat{k},\dots, n\}$ and for each $u\in \{1,2,...,N\}$;
\item $\frac{\alpha \prod_{l=1}^{c}d_l}{b_k(\prod_{j=n+1}^{n+c+1}b_j)}\geq \sum_{u=1}^{s} \frac{r_u-\sum_{i=1}^ne_{i,u}}{r_ue_{n,u}}$;
\item a general weighted complete intersection curve of multi-degrees $(d_1,..d_c)$ in $\mathbb{P}(b_{k}, b_{n+1},... b_{n+c+1})$ is irreducible;
\item $\mathbb{P}(e_{1,u},\dots,e_{n,u})$ are well-formed.
\end{enumerate}

Then $K_Y$ is nef and $\nu(Y)\geq n-1$. 
\end{thm}

By the same philosophy as in their paper, once we can establish the nefness criterion for the canonical divisor of the resulting variety $Y$ and check that $Y$ has only canonical singularities, we will have a corresponding minimal variety, which is likely to give us a new concrete example. In Section~\ref{sec 4}, we provide $79$ families of concrete minimal $3$-folds of general type in Table~\ref{tableA}, Table~\ref{tableAp}, Table~\ref{tableC}, and Table~\ref{tableC+}. All of these examples are different from those found by Iano-Fletcher as they have Picard numbers at least $2$. Moreover most of our examples have different deformation invariants from those of known ones. We also found several examples with the same deformation invariants as some known examples, with only the Picard numbers are different. Note that although our theorem is stated in a general form, for convenience we only consider some special cases for purpose of searching good examples. Here we mention several very interesting minimal $3$-folds found in our tables:
\begin{itemize}
\item[$\diamond$] The minimal models of the general weighted complete intersection of bi-degrees $(7,10)$ in $\bP(1,1,2,3,4,5)$ (see Table~\ref{tableA}, No.~10) has $p_g=2$ and $K^3=\frac{1}{3}$. This is a new example attaining minimal volumes among minimal $3$-folds of general type with $p_g\geq 2$, and they justify the sharpness of \cite[Theorem 1.4]{Chen07}. 
\item[$\diamond$] The minimal model of the general complete intersection of bi-degrees $(12,18)$ in 
$\bP(1,3,4,5,7,9)$ (see Table~\ref{tableA}, No.~3) is a minimal $3$-fold with $p_g=1$ and $K^3=\frac{1}{30}$; This example attains smallest known volume of minimal $3$-folds of general type with $p_g=1$.
\item[$\diamond$] The minimal model of the general hypersurface of degree $32$ in 
$\bP(1,1,5,8,12)$ (see Table~\ref{tableC}, No.~8) and the  minimal model of the general complete intersection of bi-degrees $(20,24)$ in 
$\bP(1,1,5,8,12,12)$ (see Table~\ref{tableC}, No.~9) are two smooth minimal $3$-folds with $p_g=7$ and $K^3=6$. Both examples lie on the first Noether line:
$K^3=\frac{4}{3}p_g-\frac{10}{3}$. These are new examples which are not birationally equivalent to those constructed by Kobayashi \cite{Kob} and by Chen and Hu \cite{C-H} and by Hu and Zhang\cite{HZ25}. 
\item[$\diamond$] The minimal model of the general hypersurface of degree $37$ in $\bP(1,1,6,9,14)$ (see Table~\ref{tableC}, No.~12) has $p_g=8$ and $K^3=\frac{15}{2}$, which lies precisely on the second Noether line. 
\end{itemize}

As in their previous paper, we also have some interesting application of Theorem~\ref{nefness 2} so that one may find infinite series of families of minimal $3$-folds of Kodaira dimension $2$ in Table ~\ref{tab kod 2}.

The structure of this article is as follows. In Section~\ref{sec 2}, we collect basic notions and preliminary results. In Section~\ref{sec 3}, we prove several different nefness criterions (i.e., Theorem~\ref{nefness 2} and Theorem~\ref{nefness 3}), which enables us to tell when a weighted blow-up of a well-formed quasi-smooth weighted complete intersection induces a minimal model with only canonical singularities. We also give details on the description of how to find these concrete examples. Section~\ref{sec 4} consists of presenting concrete examples of minimal $3$-folds with Kodaira dimension at least $2$. 

\section{Preliminaries}\label{sec 2}
Throughout we work over any algebraically closed field of characteristic $0$.

\subsection{Canonical volume }

Let $Z$ be a smooth projective variety. The {\it canonical volume} of $Z$ is defined as
$$\Vol(Z)=\lim_{m\to \infty}\frac{\dim(Z)!\ h^0(Z,mK_Z)}{m^{\dim(Z)}}.$$
For an arbitrary normal projective variety $X$, the {\it geometric genus} of $X$ is defined as $p_g(X)=p_g(Z)=h^0(Z, K_Z)$, and the {\it canonical volume} of $X$ is defined as 
$\Vol(X) = \Vol(Z),$
where $Z$ is a smooth birational model of $X$. It is known that both $p_g(X)$ and $\Vol(X)$ are independent of the choice of $Z$. Moreover, if $X$ has at worst canonical singularities and $K_X$ is nef, then $\Vol(X)=K^{\dim(X)}_X.$ 

A normal projective variety $X$ is {\it of general type} if $\Vol(X)>0$. A projective variety $X$ is {\it minimal} if $X$ is normal $\mathbb{Q}$-factorial with at worst terminal singularities and $K_X$ is nef.

\subsection{Kodaira dimension and numerical Kodaira dimension}\

\begin{defn}
Let $D$ be a $\mathbb{Q}$-Cartier $\mathbb{Q}$-divisor on a normal projective variety $X$. The {\it Kodaira dimension} of $D$ is defined to be
{\small $$
\kappa(X, D)=\begin{cases}
\max\{k\in \mathbb{Z}_{\geq 0}\mid \varlimsup\limits_{m\to \infty}m^{-k}{h^0(X,\lfloor mD\rfloor )}>0\}, & \text{if\ } |\lfloor mD\rfloor|\neq\emptyset\\
&\text{for a\ }m\in \mathbb{Z}_{> 0}; \\
-\infty, & \text{otherwise}.
\end{cases}$$}
\end{defn}

For a normal projective variety $X$ such that $K_X$ is $\mathbb{Q}$-Cartier, denote $\kappa(X)=\kappa(X, K_X)$. Note that if $X$ has at worst canonical singularities, then $\kappa(X)=\dim X$ if and only if $X$ is of general type.

\begin{defn}
Let $D$ be a nef $\mathbb{Q}$-Cartier $\mathbb{Q}$-divisor on a projective variety $X$. The {\it numerical Kodaira dimension} of $D$ is defined to be 
$$
\nu(X, D):=\max\{k\in \mathbb{Z}_{\geq 0}\mid D^k\not\equiv 0\}.
$$
\end{defn}

For a normal projective variety $X$ such that $K_X$ is $\mathbb{Q}$-Cartier and nef, denote $\nu(X)=\kappa(X, K_X)$. 
The famous abundance conjecture states that, if $X$ is a normal projective variety with mild singularities (e.g., canonical singularities) such that $K_X$ is $\mathbb{Q}$-Cartier and nef, then $\kappa(X)=\nu(X)$. In particular, this conjecture was proved if $\dim X=3$ and $X$ has canonical singularities (see \cite{K5, Mi4, K6} and references therein).

\subsection{Cyclic quotient singularities}\label{sec quot sing}\

Let $r$ be a positive integer. Denote by ${\bm \mu}_r$ the cyclic group of $r$-th roots of unity in $\bC$. A {\em cyclic quotient singularity} is of the form $\mathbb{A}^{n}/{\bm \mu}_r$, where the action of ${\bm \mu}_r$ is given by 
\[
{\bm \mu}_r\ni \xi: (x_1, \dots, x_n)\mapsto (\xi^{a_1}x_1, \dots,\xi^{a_n} x_n)
\]
for certain $a_1, \dots, a_n\in \mathbb{Z}/r$. Note that we may always assume that the action of ${\bm \mu}_r$ on $\mathbb{A}^{n}$ is small, that is, it contains no reflection (\cite[Definition 7.4.6, Theorem 7.4.8]{Ishii}), which is equivalent to that $\gcd(r, a_1,...,\hat{a}_i,...,a_n)=1$ for every $1\leq i\leq n$ by \cite[Remark 1]{Fujiki}. In this case, we say that $\mathbb{A}^{n}/{\bm \mu}_r$ is of {\em type $\frac{1}{r}(a_1,\dots, a_n)$}. We say that $Q\in X$ is a {\em cyclic quotient singularity of type $\frac{1}{r}(a_1,\dots, a_n)$}
if $(Q\in X)$ is locally analytically isomorphic to a neighborhood of $(0\in \mathbb{A}^{n}/{\bm \mu}_r)$. Recall that this singularity is isolated if and only if $\gcd(a_i, r)=1$ for every $1\leq i\leq n$ by \cite[Remark 1]{Fujiki}. 

The toric geometry interpretation of cyclic quotient singularities, by virtue of Reid \cite[(4.3)]{Rei87}, is as follows. Let $\overline{M}\simeq \mathbb{Z}^{n}$ be the lattice of monomials on $\mathbb{A}^{n}$, and $\overline{N}$ its dual. Define $N=\overline{N}+\mathbb{Z}\cdot \frac{1}{r}(a_1, \dots, a_n)$ and $M\subset \overline{M}$ the dual sub-lattice. Let $\sigma=\mathbb{R}_{\geq 0}^{n}\subset N_\mathbb{R}$ be the positive quadrant and $\sigma^\vee \subset M_\mathbb{R}$ the dual quadrant. Then in the language of toric geometry, 
\[
\mathbb{A}^{n}=\text{Spec}~\bC[\overline{M}\cap \sigma^\vee]
\]
and its quotient
\[
\mathbb{A}^{n}/{\bm \mu}_r=\text{Spec}~\bC[{M}\cap \sigma^\vee]=T_N(\Delta),
\]
where $\Delta$ is the fan corresponding to $\sigma$.

We refer to \cite{Rei87} for the definitions of terminal singularities and canonical singularities. An important fact is that terminal $3$-fold singularities are always isolated. Another important fact that we shall frequently use is a criterion on whether a cyclic quotient singularity is terminal or canonical.

\begin{lem}[{\cite[4.11]{Rei87}}]\label{can lem}
A cyclic quotient singularity of type $\frac{1}{r}(a_1,\dots, a_n)$ is terminal (resp. canonical) if and only if 
$$
\sum_{i=1}^n\big\{\frac{ka_i}{r}\big\}> 1 \, (\text{resp. } \geq 1)
$$
 for $k=1,\dots, r-1$. Here $\big\{\frac{ka_i}{r}\big\}=\frac{ka_i}{r}-\rounddown{\frac{ka_i}{r}}$ for each $i$ and $k$. 
\end{lem}

Also it is well-known that a $3$-dimensional cyclic quotient singularity is terminal if and only if it is of type 
$\frac{1}{r}(1,-1, a)$ with $\gcd(a,r)=1$ (\cite[5.2]{Rei87}).

\subsection{The Reid basket}\

A {\it basket} $B$ is a collection of pairs of integers (permitting
weights), say $\{(b_i,r_i)\mid i=1, \dots, s;\ \gcd(b_i, r_i)=1\}$. For simplicity, we will alternatively write a basket as a set of pairs with weights, 
 say for example,
$$B=\{(1,2), (1,2), (2,5)\}=\{2\times (1,2), (2,5)\}.$$

Let $X$ be a $3$-fold with at worst canonical singularities. According to 
Reid \cite{Rei87}, there is a basket of terminal cyclic quotient singularities (called the {\it Reid basket})
$$B_X=\bigg\{(b_i,r_i)\mid i=1,\dots, s;\ 0<b_i\leq \frac{r_i}{2};\ \gcd(b_i,r_i)=1\bigg\}$$
associated to $X$, where a pair $(b_i,r_i)$ corresponds to a terminal cyclic quotient singularity of type $\frac{1}{r_i}(1,-1,b_i)$. The way of determining the Reid basket of $X$ is to take a terminalization (i.e. a crepant $\mathbb{Q}$-factorial terminal model) $X'\to X$ and to locally deform every terminal singularity of $X'$ into a finite set of terminal cyclic quotient singularities.

In this article we only need to compute the baskets for minimal projective $3$-folds with terminal cyclic quotient singularities, in this case the Reid basket coincides with the set of singular points, where a singular point of type $\frac{1}{r}(1,-1,a)$ is simply denoted as $(a,r)$ under no circumstance of confusion. 

Reid's plurigenus formula is a powerful tool for computing the plurigenus for a canonical $3$-fold once we know precisely its corresponding baskets of singularites. We state the result for minimal projective terminal $3$-folds.

\begin{defn}
     For a singularity \( Q \) of type \(\frac{1}{r}(1, -1, b)\) define:

\[l(Q, n) = 
\begin{cases} 
0 & \text{if } n = 0, 1, \\
\sum_{k=1}^{n-1} \frac{\overline{bk}(r - \overline{bk})}{2r} & \text{if } n \geq 2,
\end{cases}\]

where \(\overline{x}\) denotes the smallest nonnegative residue of \(x\) modulo \(r\). This is extended to negative integers via:

\[l(-n) = -l(n + 1)\]

for all \(n \geq 0\). This is for consistency with Serre duality. For a collection (or basket) \(B\) of singularities define:

\[l(n) = \sum_{Q \in B} l(q, n)\]

for all \(n \in \mathbb{Z}\).
\end{defn}

\begin{lem}[{\cite[10.3]{Rei87}}]\label{pluri formula}
 For any projective 3-fold \( X \), with at worst terminal singularities, with  $B_X$ being its basket of singularities, then we have:

\[\chi(\mathcal{O}_X(nK_X)) = \frac{(2n-1)n(n-1)}{12} K_X^3 - (2n-1)\chi(\mathcal{O}_X) + l(n)\]

for all \( n \in \mathbb{Z} \).

 In particular, if $X$ is a minimal $3$-fold of general type, then $K_X$ is big and nef, we have:

 \[P_n = \chi(\mathcal{O}_X(nK_X)) = \frac{(2n-1)n(n-1)}{12}K_X^3 - (2n-1)\chi(\mathcal{O}_X) + l(n)\]

for all \( n \geq 2 \).
\end{lem}

\subsection{Weighted projective spaces and weighted complete intersections}\label{subsec wps}\

We refer to \cite{WPS, Fle00} for basic knowledge of weighted projective spaces and weighted complete intersections.
\begin{defn}[{\cite[5.11, 6.10]{Fle00}}]\label{wellform}
	\begin{enumerate}
		\item A weighted projective space $\mathbb{P}(a_0,...,a_n)$ is {\it well-formed} if $\gcd(a_0,...,\hat{a}_i,...,a_n)=1$ 
			for each $i$.
		\item A weighted complete intersection $X_{d_1,...d_c}$ in $\mathbb{P}(a_0,...,a_n)$ of multi-degrees $(d_1,...,d_c)$ is {\it well-formed} if $\mathbb{P}(a_0,...,a_n)$ is  well-formed and doesn't contain any codimension $c+1$ singular stratum of $\mathbb{P}(a_0,...,a_n)$.
	\end{enumerate}
\end{defn}

\begin{defn}[{\cite[3.1.5]{WPS}, \cite[6.1]{Fle00}}]\label{qsms}
	A weighted complete intersection $X\subset \mathbb{P}(a_0,...,a_n)$ is {\em quasi-smooth} if the corresponding affine cone of $X$ in $\mathbb{A}^{n+1}$ is smooth outside the origin point.
\end{defn}

We state a general result about a sufficient and necessary criteria about the quasi-smoothness of a general weighted complete intersection. Although we only need the result in the special case where $c\leq 3$, the following result gives a complete generalization of the previous partial results in \cite[Chapter 8]{Fle00}.

\begin{prop}[{cf. \cite[Proposition 3.1]{PST17}}]\label{2.8}
	Let \( X = X_{d_1, \ldots, d_c} \subset \mathbb{P}(a_0, \ldots, a_n) \) be a quasi-smooth weighted complete intersection which is not a linear cone (that is, $d_i\neq a_j$ for any $i,j$). For a subset \( I = \{i_1, \ldots, i_k\} \subset \{0, \ldots, n\} \) let \(\rho_I = \min\{c, k\}\), and for a \(k\)-tuple of natural numbers \(m = (m_1, \ldots, m_k)\) write \(m \cdot a_I := \sum_{j=1}^k m_j a_{i_j}\). Then, one of the following conditions holds.

(Q1) There exist distinct integers \(p_1, \ldots, p_{\rho_I} \in \{1, \ldots, c\}\) and \(k\)-tuples \(M_1, \ldots, M_{\rho_I} \in \mathbb{N}^k\) such that \(M_j \cdot a_I = d_{p_j}\) for \(j = 1, \ldots, \rho_I\).

(Q2) Up to a permutation of the degrees, there exist:

- an integer \(l < \rho_I\),

- integers \(e_{\mu,r} \in \{0, \ldots, n\} \setminus I\) for \(\mu = 1, \ldots, k-l\) and \(r = l+1, \ldots, c\),

- \(k\)-tuples \(M_1, \ldots, M_l\) such that \(M_j \cdot a_I = d_j\) for \(j = 1, \ldots, l\),

- for each \(r\), \(k\)-tuples \(M_{\mu,r}, \mu = 1, \ldots, k-l\) such that \(a_{e_{\mu,r}} + M_{\mu,r} \cdot a_I = d_r\),

satisfying the following property: for any subset \(J \subset \{l+1, \ldots, c\}\),

\[| \{e_{\mu,r} : r \in J, \mu = 1, \ldots, k-l\}| \geq k - l + |J| - 1.\]

Conversely, if for all subsets \(I \subset \{0, \ldots, n\}\) either (Q1) or (Q2) holds, then a general WCI \(X_{d_1, \ldots, d_c} \subset \mathbb{P}(a_0, \ldots, a_n)\) is quasi-smooth.
\end{prop}

  We also state some basic but useful properties about well-formed and quasi-smooth weighted complete intersections, all of them can be found in \cite{ Fle00}.

  \begin{prop}
       Let \( X = X_{d_1, \ldots, d_c} \subset \mathbb{P} = \mathbb{P}(a_0, \ldots, a_n) \) be a well-formed quasi-smooth weighted complete intersection.

(i) \(\text{Sing}(X) = X \cap \text{Sing}(\mathbb{P})\).

(ii) If \(\dim X > 2\), then \(\text{Cl}(X) \cong \mathbb{Z}\) and is generated by \(\mathcal{O}_X(1) := \mathcal{O}_{\mathbb{P}}(1)|_X\).

(iii) Adjunction holds: in particular, if \(\dim X > 2\) the canonical sheaf of \(X\) is given by

\[K_X = \mathcal{O}_X\left(\sum_{i=1}^c d_i - \sum_{j=0}^n a_j\right).\]

The integer number \(\alpha = \sum_{i=1}^c d_i - \sum_{j=0}^n a_j\) is called the amplitude of \(X\). In particular, the self intersection number $(\mathcal{O}_X(K_X))^{\dim X} = \frac{\alpha^{\dim X}\prod_{i=1}^{c}d_i}{\prod_{j=0}^{n}a_j}$.

(iv) The space of global sections of \(\mathcal{O}_X(k)\) can be computed from the homogeneous coordinate ring of \(X\). More precisely, let \(A = \mathbb{C}[x_0, \ldots, x_n]/(f_1, \ldots, f_c)\) be the homogeneous coordinate ring of \(X\), \(A_k\) its \(k\)-graded part. Then,

\[H^0(X, \mathcal{O}_X(k)) \simeq A_k.\]

(v) When $\dim X>2$ and $X$ is not a linear cone, then $X$ is quasi-smooth will imply $X$ is also well-formed.
\end{prop}

\subsection{Singularities on weighted complete intersections}

Singularities of a well-formed quasi-smooth weighted complete intersection can be determined by looking at its defining equations. We refer to \cite[Sections 9-10]{Fle00} for the general method. Here we illustrate the result on singularities of $3$-dimensional weighted complete intersections with codimension at most $2$.

\begin{prop}[The Hypersurface case, {cf. \cite[Sections 9-10]{Fle00}}]\label{non iso can}
Let $X_d$ be a general well-formed quasi-smooth $3$-dimensional hypersurface in $\mathbb{P}(a_0,...,a_4)$ of degree $d$. Suppose that $\gcd(a_i, a_j, a_k)=1$ for any distinct $0\leq i,j,k\leq 4$. Then the singularities of $X_d$ only arise along the edges and vertices of $\mathbb{P}(a_0,...,a_4)$. Denote $P_0, \dots, P_4$ to be the vertices. Then the set of singularities of $X_d$ is determined as follows:
\begin{enumerate}
 \item For a vertex $P_i$,
 \begin{enumerate}[label=(1.\roman*)]
 \item if $a_i | d$, then $P_i\not \in X_d$;
 \item if $a_i\nmid d$, then there exists another index $j$ such that $a_i | d-a_j$, and $P_i\in X_d$ is a cyclic quotient singularity of type $\frac{1}{a_i}(a_k, a_l, a_m)$.
 \end{enumerate}
 \item For an edge $P_iP_{j}$ (that is, $\overline{P_iP_{j}}\setminus \{P_i, P_j\}$, where $\overline{P_iP_{j}}$ is the line passing through $P_i$ and $P_j$), denote $e=\gcd(a_i, a_j)$,
 \begin{enumerate}[label=(2.\roman*)]
 \item if $e | d$, then $P_iP_{j}\cap X$ consists of exactly $\lfloor \frac{ed}{a_ia_j}\rfloor$ points, each point is a cyclic quotient singularity of type $\frac{1}{e}(a_k, a_l, a_m)$;
\item if $e\nmid d$, then $P_iP_{j}\subset X$, and there exists another index $k$ such that $e | d-a_k$, in this case, $P_iP_{j}$ is analytically isomorphic to $\mathbb{C}^*\times \frac{1}{e}(a_l, a_m)$, and each point on $P_iP_{j}$ is a cyclic quotient singularity of type $\frac{1}{e}(0, a_l, a_m)$. 
 \end{enumerate}
\end{enumerate}
Here $\{i,j,k,l,m\}$ is a reordering of $\{0,1,2,3,4\}$.
\end{prop}

\begin{prop}[The Codimension 2 case, {cf. \cite[Sections 9-10]{Fle00}}]\label{non iso can 2}
Let $X_{d_1,d_2}$ be a general well-formed quasi-smooth $3$-dimensional hypersurface in $\mathbb{P}(a_0,...,a_5)$ of bi-degrees $(d_1,d_2)$. Suppose that $\gcd(a_i, a_j, a_k, a_l)=1$ for any four distinct indices $0\leq i,j,k,l\leq 5$. Then the singularities of $X_{d_1,d_2}$ only arise along the stratums of $\mathbb{P}(a_0,...,a_5)$ of dimension at most 2. Denote $P_0, \dots, P_5$ to be the vertices. Then the set of singularities of $X_{d_1,d_2}$ is determined as follows:
\begin{enumerate}
 \item For a vertex $P_i$,
 \begin{enumerate}[label=(1.\roman*)]
 \item if $a_i | d_1$ or $a_i|d_2$, then $P_i\not \in X_{d_1,d_2}$;
 \item if $a_i\nmid d_1$ and $a_i\nmid d_2$, then there exists two other indies $j,s$ such that $a_i | d_1-a_j$ and $a_i | d_2-a_s$ and $P_i\in X_{d_1,d_2}$ is a cyclic quotient singularity of type $\frac{1}{a_i}(a_k, a_l, a_m)$.
 \end{enumerate}
 \item For an edge $P_iP_{j}$ (that is, $\overline{P_iP_{j}}\setminus \{P_i, P_j\}$, where $\overline{P_iP_{j}}$ is the line passing through $P_i$ and $P_j$), denote $e=\gcd(a_i, a_j)$,
 \begin{enumerate}[label=(2.\roman*)]
 \item if $e | d_1$ and $e | d_2$, then $P_iP_{j}\cap X = \emptyset$;
 \item if $e | d_1$ but $e\nmid d_2$, then there exists another index $s$ such that there exists monomial of the form $x_i^{n_i}x_j^{n_j}x_s$ of degree $d_2$, and $P_iP_{j}\cap X$ consists of exactly $\lfloor \frac{ed_1}{a_ia_j}\rfloor$ points, each point is a cyclic quotient singularity of type $\frac{1}{e}(a_k, a_l, a_m)$;
 \item if $e | d_2$ but $e\nmid d_1$, then there exists another index $s$ such that there exists monomial of the form $x_i^{n_i}x_j^{n_j}x_s$ of degree $d_1$, and $P_iP_{j}\cap X$ consists of exactly $\lfloor \frac{ed_2}{a_ia_j}\rfloor$ points, each point is a cyclic quotient singularity of type $\frac{1}{e}(a_k, a_l, a_m)$;
\item if $e\nmid d_1$ and $e\nmid d_2$, then $P_iP_{j}\subset X$, and there exists two other indices $k,s$ such that $e | d_1-a_k$ and $e | d_2-a_s$, in this case, $P_iP_{j}$ is analytically isomorphic to $\mathbb{C}^*\times \frac{1}{e}(a_l, a_m)$, and each point on $P_iP_{j}$ is a cyclic quotient singularity of type $\frac{1}{e}(0, a_l, a_m)$. 
 \end{enumerate}
 \item For a face $P_iP_jP_k$, denote $e= \gcd(a_i,a_j,a_k)$, we know from well-formedness that $e| d_1$ or $e| d_2$.
 \begin{enumerate}[label=(3.\roman*)]
  \item if $e| d_1$ and $e| d_2$, then $P_iP_jP_k\cap X$ consists of finitely many points, each point is a cyclic quotient singularity of type $\frac{1}{e}(a_l, a_m, a_s)$. The number of such points $N$ can be computed as follows: $N = \frac{d_1d_2e}{a_ia_ja_k}-\sum_{i,j,k}\frac{n_te}{a_t} -\sum_{p<q}\frac{n_{p,q}e}{gcd(a_p,a_q)}$, where $n_t \in \{0,1\}$ with $n_t = 1$ if and only if $P_t \in X$, $n_{p,q}$ denotes the number of points in $P_{p}P_{q}\cap X$, $t,p,q$ are indices in $\{i,j,k\}$.
  \item if $e| d_1$ but $e\nmid d_2$, then $P_iP_jP_k\cap X$ is an one dimensional subset, and there exists another index $s$ such that $e|d_2-a_s$, in this case the singularity points are not isolated and each point is a cyclic quotient singularity of the type $\frac{1}{e}(0, a_l, a_m)$.
  \item if $e| d_2$ but $e\nmid d_1$, then $P_iP_jP_k\cap X$ is an one dimensional subset, and there exists another index $s$ such that $e|d_1-a_s$, in this case the singularity points are not isolated and each point is a cyclic quotient singularity of the type $\frac{1}{e}(0, a_l, a_m)$.
\end{enumerate}
\end{enumerate}
Here $\{i,j,k,l,m,s\}$ is a reordering of $\{0,1,2,3,4,5\}$.
\end{prop}

  Of course, one may give a complete examination of the singularities of weighted complete intersections with higher codimensions. However, such a general statement would be tedious and not very useful in practice. For our interest, since the resulting minimal variety has only terminal 3-fold singularities, we want the corresponding weighted complete intersection to have only isolated or one dimensional singularities. In this case, we will have stronger restrictions on the degrees and the weights, which enables us to determine the corresponding singularities more easily.

\subsection{Weighted blow-ups of cyclic quotient singularities}

Weighted blow-ups of cyclic quotient singularities play important roles in our construction of new examples. As cyclic quotient singularities are toric, certain blow-ups can be constructed and computed easily using toric geometry. We recall the following proposition from \cite{And18}.
\begin{prop}[{cf. \cite{And18}}]\label{wb}
	Let $X$ be a normal variety of dimension $n$ such that $K_X$ is $\mathbb{Q}$-Cartier.
Suppose that $X$ has a cyclic quotient singularity $Q$ of type $\frac{1}{r}(a_1,a_2,\dots, a_n)$, where $a_1>0$, $a_2>0$,$\cdots$, $a_n>0$, $\gcd(a_1,a_2,\dots, a_n)=1$. Then we can take a weighted blow-up $\pi:Y\to X$, at $Q$ with weight $(a_1,a_2,\dots, a_n)$, which has the following properties:
\begin{enumerate}
 \item The exceptional divisor $\pi^{-1}(Q)=E\cong \mathbb{P}(a_1,a_2,\dots, a_n)$.
 \item $\OO_Y(E)|_E\cong \OO_{\mathbb{P}(a_1,a_2,\dots, a_n)}(-r)$.
 \item Locally over $Q$, $Y$ is covered by $n$ affine pieces of cyclic quotient singularities of types $\frac{1}{a_i}(-a_1,\dots, r,\dots, -a_n)$, which is obtained by replacing the $i$-th term of $(-a_1,\dots, -a_n)$ with $r$ for each $i$. 
 \item $K_Y=\pi^*(K_X)-\frac{r-\sum_{i=1}^n a_i}{r}E$.
\end{enumerate}
In particular, if $X$ is projective and $\mathbb{P}(a_1,a_2,\dots, a_n)$ is well-formed, then
$$
K_Y^n=\Big(\pi^*(K_X)-\frac{r-\sum_{i=1}^n a_i}{r}E\Big)^n=K_X^n-\frac{(r-\sum_{i=1}^n a_i)^n}{r\prod_{i=1}^na_i}.
$$
\end{prop}

\begin{exmp}
    As explained in the introduction, we are often required to perform a weighted blow up at an isolated non-canonical cyclic quotient singular point on the weighted complete intersection. It's natural to ask when such one weighted blow up won't give some new non-canonical singular points. In practice it's an easy check, but we must be aware that there exists many cases that just one weighted blow up doesn't suffice.
    \begin{enumerate}
        \item $\frac{1}{7}(1,2,3)$ is a non-canonical singularity point such that after performing a weighted blow up at this point with weights $(1,2,3)$, the resulting variety will have two new singularities $\frac{1}{2}(1,1,1)$ and $\frac{1}{3}(1,2,1)$, which are both indeed terminal.
        \item $\frac{1}{11}(1,1,4)$ is a non-canonical singularity point such that after performing a weighted blow up at this point with weights $(1,1,4)$, the resulting variety will have one new singular point $\frac{1}{4}(1,1,1)$, which is still non-canonical. However if we continue performing another weighted blow up at this point, we will result in a variety with only terminal singularities.
    \end{enumerate}
\end{exmp}

\section{The nefness criterion}\label{sec 3}
 
 In this section, we start by looking into a well-formed quasi-smooth weighted complete intersection $X$ with isolated singularities, among which it either has only one non-canonical singular point, or it has several non-canonical points in special positions. If we take a partial resolution, by means of a weighted blow-up at those singular points on $X$, to get the birational morphism $Y \to X$, $Y$ would have milder singularities than $X$ does. A very natural question is whether $K_Y$ is nef. If so, this will give us an easy but new strategy for constructing minimal varieties. A key observation of this section is Theorem~\ref{nefness 2} and Theorem~\ref{nefness 3} (which we simply call the ``nefness criterion''). As applications of Theorem~\ref{nefness 2} and Theorem~\ref{nefness 3}, in the next section, we 
 construct 79 families of new minimal $3$-folds of general type, in Table~\ref{tableA}, Table~\ref{tableAp}, Table~\ref{tableC}, and Table~\ref{tableC+}, followed by infinite families of minimal $3$-folds of Kodaira dimension $2$ (see Tables~\ref{tab kod 2}).

\begin{proof}[Proof of Theorem~\ref{nefness 2}]
Recall that $X$ has a cyclic quotient singularity of type $\frac{1}{r}(e_1,\dots, e_n)$ at $Q$ and $\pi: Y\to X$ be the weighted blow-up at $Q$ with weight $(e_1,\dots, e_n)$ as in Proposition~\ref{wb}.

We first show the case where there exists some index $k$ satisfying the assumptions (1) to (4). Without loss of generality, after rearranging of indices, we may assume that the assumptions hold for $ k=n$.
For each $j=1,\dots, n-1$, let $H_j\subset X$ be the effective Weil divisor defined by $x_j=0$, denote $L$ to be a Weil divisor corresponding to $\mathcal{O}_X(1)$. Then $$H_j\sim b_j L\ \text{and}\ K_X\sim \alpha L.$$
Denote $H'_j$ to be the strict transform of $H_j$ on $Y$ and $E$ to be the exceptional divisor of $\pi$. Then 
$$\pi^*H_j=H'_j+\frac{e_j}{r}E.$$ 
Denote $t_j=\frac{\alpha e_j-b_j(r-\sum_{i=1}^ne_i)}{b_jr}$. Then $t_j\geq 0$ for each $j=1,\dots, n-1$ by assumption.
As $K_Y=\pi^*K_X-\frac{r-\sum_{i=1}^ne_i}{r}E$, 
we can see that 
\begin{align}\label{K=H+E}
K_Y\sim_\mathbb{Q} \frac{\alpha}{b_j}H'_j+t_j E
\end{align}
for each $j=1,\dots, n-1$.

Assume, to the contrary, that $K_Y$ is not nef. Then there exists a curve $C$ on $Y$ such that $(K_Y\cdot C)<0$.
Note that $K_Y|_E=\frac{r-\sum_{i=1}^ne_i}{r}(-E)|_E$ is ample, hence $C\not \subset E$. Therefore Equation \eqref{K=H+E} implies that $C\subset \cap_{j=1}^{n-1} H'_j$. 

We claim that $\text{Supp}(\cap_{j=1}^{n-1} H'_j)=C$. It suffices to show that $\text{Supp}(\cap_{j=1}^{n-1} H'_j)$ is an irreducible curve.
Note that $\pi (\text{Supp}(\cap_{j=1}^{n-1} H'_j))=\cap_{j=1}^{n-1} H_j$ is a general weighted complete intersection of multi-degrees $(d_1,d_2,...d_c)$ in $\mathbb{P}(b_n, b_{n+1}, ...b_{n+c+1}),$ hence $\pi (\text{Supp}(\cap_{j=1}^{n-1} H'_j))$ is an irreducible curve by assumption. On the other hand, the support of $\cap_{j=1}^{n-1} H'_j\cap E$ is just the point $[0:\dots : 0: 1]$ in $E\simeq \mathbb{P}(e_1,\dots, e_n)$. So $\text{Supp}(\cap_{j=1}^{n-1} H'_j)$ is just the strict transform of $\pi (\text{Supp}(\cap_{j=1}^{n-1} H'_j))$, which is an irreducible curve.

Therefore, we can write $(H'_1\cdot \dots \cdot H'_{n-1})=t C$ for some $t>0$ as $1$-cycles. Then $(K_Y\cdot C)<0$ implies that 
$(K_Y\cdot H'_1\cdot \dots \cdot H'_{n-1} )<0$. On the other hand, 
\begin{align*}{}&(K_Y\cdot H'_1\cdot \dots \cdot H'_{n-1} )\\={}&((\pi^*K_X-\frac{r-\sum_{i=1}^ne_i}{r}E)\cdot (\pi^*H_1- \frac{e_1}{r}E)\cdot \dots \cdot (\pi^*H_{n-1}- \frac{e_{n-1}}{r}E))\\
={}&\alpha (\prod_{j=1}^{n-1}b_j ) L^{n}+(-1)^n\frac{(r-\sum_{i=1}^ne_i)\prod_{j=1}^{n-1}e_j}{r^n}E^n\\
={}&\frac{\alpha \prod_{l=1}^cd_l}{b_n\prod_{j=n+1}^{n+c+1}b_j}-\frac{r-\sum_{i=1}^ne_i}{re_n}\geq 0,
\end{align*}
a contradiction.

Next we consider the case where $\alpha e_{j} \geq b_{j}(r-\sum_{i=1}^ne_i)$ holds for all $j\in \{1, \dots, n\}$. Notation as above, in this situation we have $K_Y\sim_\mathbb{Q} \frac{\alpha}{b_j}H'_j+t_j E$ for all $j=1,\dots, n$. Now if $K_Y$ is not nef, the curve $C$ on $Y$ such that $(K_Y,C)\leq 0$ must be contained in $ \cap_{j=1}^{n} H'_j$. However, since the image of $ \cap_{j=1}^{n} H'_j$ in $X$ is just $Z_{d_1,d_2,...d_c} \subset \mathbb{P}( b_{n+1},... b_{n+c+1})$, which is a finite set by assumption, which in turn implies that $C$ is inside the exceptional locus $E$, this contradicts to the fact that $C\not \subset E$.

Hence we conclude that $K_Y$ is nef. The fact that $\nu(Y)\geq n-1$ follows from $(K_Y^{n-1}\cdot E)>0$ as $K_Y|_E$ is ample.
\end{proof}

\begin{rem}\label{nefrk}
	Here we mention a special but important case of Theorem~\ref{nefness 2}. If $\alpha=r-\sum_{i=1}^ne_i$, and there exists an index $k\in \{1,\dots,n\}$ such that 
 $b_j=e_j$ for each $j\in \{1, \dots, n\}\setminus \{k\}$, then condition (1) in Theorem~\ref{nefness 2} automatically holds and, meanwhile, condition (2) is equivalent to $K_Y^n\geq 0$ as $K_Y^n = \frac{\alpha^n \prod_{j=1}^cd_j}{\prod_{j=1}^{n+2}b_j}-\frac{(r-\sum_{i=1}^n e_i)^n}{r\prod_{i=1}^ne_i}$ by Proposition~\ref{wb}.
\end{rem}

\begin{proof}[Proof of Theorem~\ref{nefness 3}]
The idea for the proof is essentially the same as Theorem~\ref{nefness 2}. We first deal with \textbf{Case 1}. For each $j=1,\dots, n-1$, let $H_j\subset X$ be the effective Weil divisor defined by $x_j=0$ and denote by $L$ a Weil divisor corresponding to $\mathcal{O}_X(1)$. Then $H_j\sim b_j L$ and $K_X\sim \alpha L$. 
Denote by $H'_j$ the strict transform of $H_j$ on $Y$ and by $E_t$ the exceptional divisors of $\pi$ over those $Q_t$. Then 
$$\pi^*H_j=H'_j+\sum_{t=1}^s\frac{e_{j,t}}{r_t}E_t.$$ 
Set $\lambda_{j,t}=\frac{\alpha e_{j,t}-b_j(r_t-\sum_{i=1}^ne_{i,t})}{b_jr_t}$. Then, for each $j=1,\dots, n-1$ and $t=1,\dots, s$, $\lambda_{j,t}\geq 0$ by assumption.
As $K_Y=\pi^*K_X-\sum_{t=1}^s\frac{r_t-\sum_{i=1}^ne_{i,t}}{r_t}E_t$ and $K_X\sim \frac{\alpha}{b_j}H_j$, 
we can see that 
\begin{align}\label{K=H+E 2}
K_Y\sim_\mathbb{Q} \frac{\alpha}{b_j}H'_j+ \sum_{t=1}^s\lambda_{j,t} E_t
\end{align}
for each $j=1,\dots, n-1$.

Assume, to the contrary, that $K_Y$ is not nef. Then there exists a curve $C$ on $Y$ such that $(K_Y\cdot C)<0$.
Note that $K_Y|_{E_t}=\frac{r_t-\sum_{i=1}^ne_{i,t}}{r_t}(-E_t)|_{E_t}$ are ample, hence $C\not \subset \cup^s_{t=1} E_t$. Therefore Equation \eqref{K=H+E 2} implies that $C\subset \cap_{j=1}^{n-1} H'_j$. 

We claim that $\text{Supp}(\cap_{j=1}^{n-1} H'_j)=C$. It suffices to show that $\text{Supp}(\cap_{j=1}^{n-1} H'_j)$ is an irreducible curve.
Note that $\pi (\text{Supp}(\cap_{j=1}^{n-1} H'_j))=\cap_{j=1}^{n-1} H_j$ is a general weighted complete intersection of multi-degrees $(d_1,...,d_c)$ in $\mathbb{P}(b_n,..., b_{n+c+1}),$ hence $\pi (\text{Supp}(\cap_{j=1}^{n-1} H'_j))$ is an irreducible curve by assumption. On the other hand, the support of $\cap_{j=1}^{n-1} H'_j\cap E_t$ is just the point $[0:\dots : 0: 1]$ in $E_t\simeq \mathbb{P}(e_{1,t},\dots, e_{n,t})$. So $\text{Supp}(\cap_{j=1}^{n-1} H'_j)$ is just the strict transform of $\pi (\text{Supp}(\cap_{j=1}^{n-1} H'_j))$, which is an irreducible curve.

Therefore, we can write $(H'_1\cdot \dots \cdot H'_{n-1})=t C$ for some $t>0$ as $1$-cycles. Then $(K_Y\cdot C)<0$ implies that 
$(K_Y\cdot H'_1\cdot \dots \cdot H'_{n-1} )<0$. On the other hand, 
{\begin{align*}{}&(K_Y\cdot H'_1\cdot \dots \cdot H'_{n-1} )\\
={}&{\big(}(\pi^*K_X-\sum_{t=1}^s\frac{r_t-\sum_{i=1}^ne_{i,t}}{r_t}E_t)\cdot (\pi^*H_1- \sum_{t=1}^s\frac{e_{1,t}}{r_t}E_t)\\
{}&\cdot \cdots \cdot (\pi^*H_{n-1}- \sum_{t=1}^s\frac{e_{n-1,t}}{r_{t}}E_t){\big)}\\
={}&\alpha (\prod_{j=1}^{n-1}b_j ) L^{n}+ \sum_{t=1}^s(-1)^n\frac{(r_t-\sum_{i=1}^ne_{i,t})\prod_{j=1}^{n-1}e_{j,t}}{r_t^n}E_t^n\\
={}&\frac{\alpha \prod_{k=1}^c d_k}{\prod_{j=n}^{n+c+1}b_j}-\sum_{t=1}^s\frac{r_t-\sum_{i=1}^ne_{i,t}}{r_te_{n,t}}\geq 0,
\end{align*}}
a contradiction.

For \textbf{Case 2}, the situation is similar except in this case all the non-canonical points are contained in one linear subspace $\Pi$, so we can choose one index $k\in \{1,\dots,n\}$ and let the rest indices satisfying the properties, so that $K_Y$ is $\mathbb{Q}$-linearly equivalent to the strict transform of $H_j$ plus some effective sum of the exceptional divisors. The rest of the arguments are just the same.
\end{proof}

As we mentioned before, sometimes after only one weighted blow up the resulting variety will still have another non-canonical singular point. We now deal with this situation so that after we perform another weighted blow up, the final output variety will still preserve the nefness of the canonical divisor. This situation is rather subtle and for simplicity we only state the result for $3$-folds which is the only case we shall use in this paper.

\begin{thm}\label{nefness 4}
Let $X=X^3_{d_1,...d_c}\subset \mathbb{P}(b_1, \dots, b_{c+4})$ be an $3$-dimensional well-formed quasi-smooth general weighted complete intersection of multi-degrees $(d_1,\dots,d_c)$ with $\alpha=\sum_{i=1}^cd_i-\sum_{j=1}^{c+4}b_j>0$ where $b_1, \dots, b_{c+4}$ are not necessarily sorted by size. Denote by $x_1,\dots,x_{c+4}$ the homogenous coordinates of $\mathbb{P}(b_1, \dots, b_{c+4})$.
Denote by $\Pi$ the linear space $(x_1=x_2=x_{3}=0)$ in $\mathbb{P}(b_1, \dots, b_{c+4})$. Suppose that $X\cap\Pi$ consists of finitely many points and take $Q\in X\cap \Pi$.
Assume that $X$ has a cyclic quotient singularity of type $\frac{1}{r}(e_1,e_2,e_3)$ at $Q$ where $e_1,e_2,e_3>0$, $\gcd(e_1,e_2,e_3)=1$, $\sum_{i=1}^3e_i<r$ and that $x_1,x_2,x_3$ are also the local coordinates of $Q$ corresponding to the weights $\frac{e_1}{r},\frac{e_2}{r},\frac{e_3}{r}$ respectively. Let $p: Y\to X$ be the weighted blow-up at $Q$ with weight $(e_1,e_2, e_3)$.
Suppose that for the index $k=3$ the conditions in Theorem\ref{nefness 2} are satisfied. Thus we already know that $K_Y$ is nef. Let $H_1\subset X$ be the effective Weil divisor defined by $x_1=0$ and $H_2\subset X$ be the effective Weil divisor defined by $x_2=0$, denote by $H_{1,Y}$ and $H_{2,Y}$ their strict transforms to $Y$ and denote $E_Y$ the exceptional divisor of $p$. Assume $Y$ still has a non-canonical singular point $Q'$ located at $H_{1,Y}\cap H_{2,Y}\cap E_Y$, which is cyclic quotient of the type $\frac{1}{r'}(f_1,f_2,f_3)$ such that $f_1,f_2,f_3>0$, $\gcd(f_1,f_2,f_3)=1$, $\sum_{i=1}^3f_i<r$. Let $q: Z\to Y$ be another weighted blow-up at $Q'$ with weight $(f_1,f_2,f_3)$. Assume further that:
\begin{enumerate}
    \item $\alpha e_{j} = b_{j}(r-\sum_{i=1}^3e_i)$ for each $j\in \{1, 2\}$;
    \item $\alpha f_{j} \geq b_{j}(r'-\sum_{i=1}^3f_i)$ for each $j\in \{1, 2\}$;
    \item $\frac{\alpha \prod_{i=1}^cd_i}{\prod_{j=3}^{c+4}b_j} - \frac{r-\sum_{i=1}^3e_i}{re_3} - \frac{r'-\sum_{i=1}^3f_i}{r'f_3} \geq 0$;
    \item $\mathbb{P}(f_1,f_2,f_3)$ is well-formed.
\end{enumerate}
Then $K_Z$ is also nef and $\nu(Z)\geq 2$.
\end{thm}

\begin{proof}
    Consider the sequence of weighted blow ups $Z\to Y\to X$, denote $F_Z$ to be the exceptional divisor on $Z$ over $Y$ and $E_Z$ $H_{1,Z}$ $H_{2,Z}$ the strict transform of $E_Y$ $H_{1,Y}$ $H_{2,Y}$ to $Z$. We know that $K_Z = q^*K_Y - \frac{r'-\sum_{i=1}^3f_i}{r'}F_Z$. As in the proof of Theorem \ref{nefness 2}, we may write $K_Y\sim_\mathbb{Q} \frac{\alpha}{b_j}H_{j,Y}+t_j E_Y$ for each $j\in \{1,2\}$. We may then write
    \begin{align}\label{K=H+E+F}
K_Z\sim_\mathbb{Q} \frac{\alpha}{b_j}H_{j,Z}+t_j E_Z + (\frac{\alpha}{b_j}\frac{f_j}{r'}-\frac{r'-\sum_{i=1}^3f_i}{r'}+\frac{t_jf_3}{r'})F_Z
\end{align}
for each $j \in \{1,2\}$. Here we used the fact that $q^*H_{j,Y} = H_{j,Z}+ \frac{f_j}{r'}F_Z$ and $q^*E_Y = E_Z + \frac{f_3}{r'}F_Z$. Now by assumption, $t_j=0$ and $s_j= \frac{\alpha}{b_j}\frac{f_j}{r'}-\frac{r'-\sum_{i=1}^3f_i}{r'}\geq 0$, which implies $K_Z\sim_\mathbb{Q} \frac{\alpha}{b_j}H_{j,Z}+s_j E_Z$. Now by the same argument as in the proof of Theorem \ref{nefness 2}, we can see if $(K_Z,C)<0$ for some curve $C$, then $C\subset H_{1,Z}\cap H_{2,Z}$. In fact $\text{Supp}(H_{1,Z}\cap H_{2,Z})=C$ as its image in $X$ is an irreducible curve. Thus we have $(K_Z\cdot H_{1,Z}\cdot H_{2,Z})<0$ which contradicts to the third condition by a standard computation.
\end{proof}

\begin{rem}
    We note that the previous theorem requires $t_j=0$ so that no component $E_Z$ would appear in the expression. Of course this requirement may be too strong, but it indeed produces some interesting new examples. We also note that on $Y$ the only new non-canonical singular point is in a very special location, namely $Q'\in H_{1,Y}\cap H_{2,Y}\cap E_Y$.
\end{rem}

As the last part of this section, we provide several lemmas which are helpful for applying Theorem~\ref{nefness 2} or Theorem~\ref{nefness 3} and for computing the invariants of resulting minimal models.
	
	For verifying condition (3) of Theorem~\ref{nefness 2} and Theorem~\ref{nefness 3}, we need to check the irreducibility of a general weighted complete intersection curve in a weighted projective space, this is actually not very easy and we can only state some partial results here.

\begin{lem}\label{irreducible 1}
    Let $C$ be a general weighted complete intersection curve of multi-degrees $(d_1,...d_c)$ in a weighted projective space $\mathbb{P}(a_0,a_1,\dots ,a_{c+1})$. If $C$ is quasi-smooth, then $C$ is irreducible. If $C$ contains some one-dimensional stratum of $\mathbb{P}(a_0,a_1,\dots ,a_{c+1})$, then $C$ is not irreducible.
\end{lem}
\begin{proof}
    First, by induction on dimension and the connectedness lemma of Enrique-Severi-Zariski (see \cite[Theorem 18.12]{Eis95}) we can see that $C$ is always connected. Now since $C$ is quasi-smooth, it has only cyclic quotient singularities, which are of the type $\frac{1}{r}(a)$ for coprime integers $r$ and $a$. Let $x$ be the coordinate on $\mathbb{A}^1$, the group $\mathbb{Z}_r$ acts on $\mathbb{A}^1$ via: $x\mapsto \epsilon^ax$, where $\epsilon$ is the primitve rth root of unity. Thus $\mathbb{A}^1/\mathbb{Z}_r \cong \operatorname{Spec} K[z^r] \cong \operatorname{Spec} K[z] \cong \mathbb{A}^1$, this implies $C$ is non-singular. So in particular $C$ is irreducible since it is both smooth and connected. The last statement is trivial.
\end{proof}

\begin{rem}
    In fact, the statement of Lemma \ref{irreducible 1} holds more generally with the weighted complete intersection being positive dimensional. That is, if $X$ is a quasi-smooth positive dimensional weighted complete intersection, then $X$ is always connected. 
\end{rem}
The codimension $1$ and $2$ cases are easier to handle, as they can be realized as divisors, and irreducibility comes from the linear system being not composed of a pencil.
	
\begin{lem}\label{irreducible 2}\cite[Lemma 3.2]{CJL24}
		Let $C$ be a general hypersurface of degree $d$ in the well-formed space $\mathbb{P}(a,b,c)$. Suppose that, in the weighted polynomial ring $\bC[x,y,z]$ with $\text{\rm weight}\, x=a$, $\text{\rm weight}\, y=b$ and $\text{\rm weight}\, z=c$, 
		\begin{enumerate}
		\item there are at least two monomials of degree $d$;
		\item all monomials of degree $d$ have no common divisor;
		\item the set of monomials of degree $d$ cannot be written as $\{g_1^i g_2^{k-i}\mid i=0,1,\dots, k\}$ for some integer $k>1$, where $k$ divides $d$ and $g_1$, $g_2$ are two monomials of degree $\frac{d}{k}$.
		\end{enumerate}
		Then $C$ is irreducible. 
	\end{lem}

    \begin{lem}\label{irreducible 3}
		Let $C$ be a general weighted complete intersection curve of bi-degrees $(d_1,d_2)$ in the well-formed space $\mathbb{P}(a_0,a_1,a_2,a_3)$. Denote by $S_1=S_{d_1}\subset \mathbb{P}(a_0,a_1,a_2,a_3)$ and $S_2=S_{d_2}\subset \mathbb{P}(a_0,a_1,a_2,a_3)$ the two general weighted surfaces in $\mathbb{P}$, assume that both $S_1$ and $S_2$ are irreducible and $\mathrm{Bs}|\mathcal{O}_\mathbb{P}(d_1)|\cap \mathrm{Bs}|\mathcal{O}_\mathbb{P}(d_2)|$ is zero-dimensional. Suppose furthermore that there exists some $S_i$ such that there exists at least $2$ distinct degree $d_j$ monomials are non-zero when restricted to $S_i$, and the set of such monomials of degree $d_j$ restricted to $S_i$ cannot be written as $\{\bar{g_1}^t \bar{g_2}^{k-t}\mid t=0,1,\dots, k\}$ for some integer $k>1$, where $k$ divides $d_j$ and $g_1$, $g_2$ are two monomials of degree $\frac{d_j}{k}$. Here $\{i,j\}$ is a reordering of $\{1,2\}$.
		Then $C$ is irreducible. 
	\end{lem} 

\begin{proof}
    Suppose $i=1,j=2$. The conditions ensure that the sub-linear system $V$ of $|\mathcal{O}_{S_1}(d_2)|$ on $S_1$ generated by the restriction of degree $d_2$ monomials to $S_1$ has no base components. So if $C$ is reducible, $V$ should be composed of a pencil, which in turn yields the corresponding set of monomials of degree $d_2$ restricted to $S_1$ can be written as $\{\bar{g_1}^t \bar{g_2}^{k-t}\mid t=0,1,\dots, k\}$ for some integer $k>1$, where $k$ divides $d_2$ and $\bar{g_1}$, $\bar{g_2}$ are two monomials restricted to $S_1$ of degree $\frac{d_2}{k}$.
\end{proof}
    
	\begin{rem}\label{irreducible remark} It can be checked that condition (3) of Lemma~\ref{irreducible 2} holds in each of the following cases:
	\begin{enumerate}[label=(3.\roman*)]
	\item There exist 3 monomials of degree $d$ of forms $x^{m_1}y^{n_1}$, $y^{m_2}z^{n_2}$, $x^{m_3}z^{n_3}$, where $m_i, n_i$ are positive integers for $i=1,2,3$.
	\item There exist 2 monomials of degree $d$ of forms $x^{m_1}$, $y^{m_2}z^{m_3}$, such that $\gcd(m_1, m_2, m_3)=1$, where $m_i$ is a non-negative integer for $i=1,2,3$.
	\end{enumerate}

\end{rem}

	We also need a lemma that helps us to compute the change of Picard numbers under a crepant blow-up of a canonical cyclic quotient singularity.
	
	\begin{lem}\label{can sing res}\cite[Lemma 3.5]{CJL24}
		Let $(X,Q)$ be a germ of $n$-fold isolated cyclic quotient canonical singularity of type $\frac{1}{r}(a_1,\dots,a_n)$. Then there is a terminalization $ X'\rightarrow (X,Q)$ such that 
	 $$\rho(X'/X)=\#\{m\in \mathbb{Z}\mid \sum_{i=1}^n\big\{\frac{ma_i}{r}\big\}=1, 1\leq m \leq r-1\}.$$
	\end{lem}

\section{Applications of the nefness criterion in constructing minimal varieties}\label{sec 4}

\subsection{General construction}\label{construction process}\

 In practice, by applying Theorem~\ref{nefness 2} or Theorem~\ref{nefness 2}, it is possible to search numerous minimal varieties with the help of a computer program. The effectivity may follow from the following steps: 
\begin{quote}
Pick up a general weighted complete intersection of dimension $n$ and codimension $c$ , say
$$X=X^n_{d_1\dots d_c}\subset \mathbb{P}(a_0,a_1,...,a_{n+c}).$$
\begin{enumerate}%[itemindent=1em] 
 \item[{\bf Step 0.}] Check that $X$ is well-formed and quasi-smooth by Definition~\ref{wellform} and Theorem~\ref{qsms};
 \item[{\bf Step 1.}] Compute singularities of $X$ and check that $X$ has a unique non-canonical singularity $Q$ or has some non-canonical singularity points $Q_i$ in special position by Proposition~\ref{non iso can}.
 \item[{\bf Step 2.}] Verify that $X$ and $Q$ (or $Q_i$) satisfy the conditions of Theorem~\ref{nefness 2} or Theorem~\ref{nefness 3} using Lemma~\ref{irreducible 1} \ref{irreducible 2} or \ref{irreducible 3}, thus %by Theorem~\ref{nef pbf}, 
 one obtains a weighted blow-up $f:\widetilde{X}\rightarrow X$ at $Q$ such that $K_{\widetilde{X}}$ is nef;
 \item[{\bf Step 3.}] Compute singularities of $\widetilde{X}$ and check that $\widetilde{X}$ has only canonical singularities by Proposition~\ref{wb}, if not, check wether we are in the special situation of Theorem \ref{nefness 4} so that we can perform another weighted blow-up and compute the corresponding singularities;
 \item[{\bf Step 4.}] Take a terminalization $g:\widehat{X}\rightarrow \widetilde{X}$ by Lemma~\ref{can sing res}. 
\end{enumerate}
In the end, if $X$ passes through all above steps, then the resulting 
$\widehat{X}$ is a minimal projective $n$-fold with $\bQ$-factorial terminal singularities. 
\end{quote}

\subsection{Examples of minimal $3$-folds of general type with canonical volume less than $1$}

To find minimal $3$-folds with canonical volume less than 1, we take a general weighted complete intersection, say $X=X_{d_1,\dots d_c}\subset \mathbb{P}(a_0,a_1,...,a_{c+3})$ with $1\leq \alpha=\sum_{j=1}^cd_j-\sum_{i=0}^{c+3}a_i\leq 10$ and $5\leq d_j\leq 100$, and apply our Construction~\ref{construction process}. In practice we often require $c\leq 3$ to let the computer program run efficiently. This will output at least 79 families of minimal $3$-folds of general type, which are listed in 
Table~\ref{tableA} and Table~\ref{tableAp}. Table~\ref{tableA} consists of those $X$ with only isolated singularities, while Table~\ref{tableAp} consists of those $X$ with non-isolated singularities. 

 Here we explain the contents of the tables: each row contains a well-formed quasi-smooth general cimplete intersection $X=X_{d_1\dots d_c}\subset \mathbb{P}(a_0,a_1,...,a_{c+3}).$
	 The columns of the tables contain the following information: 
\begin{center}
	\begin{tabular}{r p{10cm}}
		\hline
		$\alpha$:& The amplitude of $X$, i.e., $\alpha=\sum d_j-\sum a_i$;\\
		$\deg$:& The degrees of $X$, which is $(d_1,\dots d_c)$;\\
		weight:& Weights of $\bP(a_0,\dots ,a_{c+3})$;\\
		B-weight(s):& $\frac{1}{r}(e_1, e_2, e_3)$, the unique non-canonical singularity in $X$, to which we apply Theorem~\ref{nefness 2} or several non-canonical singular points in$X$, to which we apply Theorem~\ref{nefness 3};\\
		$\Vol$: & the canonical volume of $\widehat{X}$, i.e., $K_{\widehat{X}}^3$;\\
		$P_2$: & $h^0(\widehat{X}, 2K_{\widehat{X}})$;\\
		$\chi$: & The holomorphic Euler characteristic of $\mathcal{O}_{\widehat{X}}$;\\
		$\rho$: & The Picard number of $\widehat{X}$;\\
		basket: & The Reid basket of $\widehat{X}$.\\		
		\hline
	\end{tabular}
\end{center}
\medskip

Here $\Vol(X)=K_{\widehat{X}}^3$ can be computed by Proposition~\ref{wb}. Note that Proposition~\ref{wb} implies that $K_{\widetilde{X}}+\frac{r-e_1-e_2-e_3}{r}E=f^*K_X$. In all listed examples, since $2(e_1+e_2+e_3)>r$, we see that, for $m=1,2$,
$$h^0(\widehat{X}, m K_{\widehat{X}})=h^0({\widetilde{X}}, mK_{\widetilde{X}})=h^0({\widetilde{X}}, \lfloor mf^*(K_{{X}})\rfloor )=h^0({{X}}, mK_{{X}})
$$
where the last item can be computed on $X$ by counting the number of monomials of degree $m\alpha$. The higher pulrigenus can be computed directly using Reid's pluri-genus formula once we know the baskets of singularities of $\widehat{X}$. Besides, we have 
$$\chi(\mathcal{O}_{\widehat{X}})=\chi(\mathcal{O}_X)=1-h^0(X, K_X)$$ by \cite[Theorem 3.2.4(iii)]{WPS}. By virtue of Proposition~\ref{non iso can}, any singularity of $\widehat{X}$ lies over a point in some stratum of $X$, which in our case (if $c\leq 2$) is just a vertex $P_i\in X$ or over some point on the line $\overline{P_iP_j}\cap X$ for some $i$ and $j$ or some point on the $2$-dimensional stratum $\overline{P_iP_jP_k}\cap X$ for some $i$ $j$ and $k$. Hence $\rho(\widehat{X})$ and the basket $B_{\widehat{X}}$
can be computed using Proposition~\ref{non iso can}, Proposition~\ref{wb}, and Lemma~\ref{can sing res}. 

All examples in Table~\ref{tableA} and Table~\ref{tableAp} have been manually verified. In fact, we found these examples by first using some necessary conditions to write a computer program to help us cut out a large number of elements from the raw lists and then check the remaining elements manually. Actually, this is not a hard work at all (see examples following the tables). 

The reason that we split into two tables is that all examples in Table~\ref{tableA} have isolated canonical singularities, so we can compute the Picard number $\rho(\widehat{X})$ of the minimal model $\widehat{X}$ easily by Lemma~\ref{can sing res}; on the other hand, all examples in Table~\ref{tableAp} have non-isolated canonical singularities, for which $\rho(\widehat{X})$ is harder to compute, so we omit the computations of $\rho(\widehat{X})$ in this table.
		
\medskip
{\tiny
\begin{longtable}{|l|l|l|l|l|l|l|l|l| p{4.5cm} |}
\caption{Minimal $3$-folds of general type, I}\label{tableA}\\
		\hline
			No.&$\alpha$ & deg & weight & B-weight & $\Vol$&$P_2$ & $\chi$&$\rho$ &basket \\ \hline	
\endfirsthead
\multicolumn{5}{l}{{ {\bf \tablename\ \thetable{}} \textrm{-- continued from previous page}}} \\
\hline 
			No.&$\alpha$ & deg & weight & B-weight & $\Vol$&$P_2$ & $\chi$&$\rho$ &basket \\ \hline	
\endhead

 \multicolumn{5}{l}{{\textrm{Continued on next page}}} \\ \hline
\endfoot
\hline 
\endlastfoot

\hline
 1 & 1 & (20, 21) & $(3,4,5,7,8,13)$ & $\frac{1}{13}(3,4,5)$ & $\frac{1}{120}$ & 0 & 1 & 2 & $(1,3)$ $3 \times (1,4)$ $(2,5)$ $(3,8)$ \\
\hline
2 & 1 & $(15,16)$ & $(2,3,4,5,5,11)$ & $\frac{1}{11}(2,3,5)$ & $\frac{1}{30}$ & 1 & 1 & 2 & $5 \times (1,2)  $ $(1,3)$  $4\times (2,5)$  \\
\hline
3 & 1 & $(12,18)$ & $(1,3,4,5,7,9)$ & $\frac{1}{7}(1,3,2)$ & $\frac{1}{30}$ & 1 & 0 & 2 & $ (1,2) $ $(1,3)$ $(1,5)$  \\
\hline
4 & 1 & $(15,16)$ & $(2,3,3,4,5,13)$ & $\frac{1}{13}(3,4,5)$ & $\frac{1}{20}$ & 1 & 1 & 2 & $4\times (1,2)$ $6\times (1,3)$ $(1,4)$ $(2,5)$ \\
\hline
5 & 1 & $(14,15)$ & $(2,3,3,4,5,11)$ & $\frac{1}{11}(2,3,5)$ & $\frac{1}{20}$ & 1 & 1 & 2 & $4\times (1,2)$ $6\times (1,3)$ $(1,4)$ $(2,5)$ \\
\hline
6 & 1 & $(12,15)$ & $(1,2,3,4,5,11)$ & $\frac{1}{11}(2,3,5)$ & $\frac{2}{15}$ & 2 & 0 & 2 & $4\times (1,2)$ $(1,3)$  $(2,5)$ \\
\hline
7 & 1 & $(11,12)$ & $(1,2,3,4,5,7)$ & $\frac{1}{7}(1,2,3)$ & $\frac{2}{15}$ & 2 & 0 & 2 & $4\times (1,2)$ $(1,3)$  $(2,5)$ \\
\hline
8 & 1 & $(10,12)$ & $(1,2,3,3,5,7)$ & $\frac{1}{7}(1,2,3)$ & $\frac{1}{6}$ & 2 & 0 & 2 & $(1,2)$ $5\times (1,3)$  \\
\hline
9 & 1 & $(9,10)$ & $(1,2,2,3,3,7)$ & $\frac{1}{7}(1,2,3)$ & $\frac{1}{3}$ & 3 & 0 & 2 & $6\times(1,2)$ $4\times (1,3)$  \\
\hline
10 & 1 & $(7,10)$ & $(1,1,2,3,4,5)$ & $\frac{1}{4}(1,1,1)$ & $\frac{1}{3}$ & 4 & -1 & 2 & $2\times(1,2)$ $(1,3)$   \\
\hline
11 & 1 & $(8,9)$ & $(1,1,2,3,4,5)$ & $\frac{1}{5}(1,1,2)$ & $\frac{1}{2}$ & 4 & -1 & 2 & $3\times(1,2)$    \\
\hline
12 & 1 & $(7,8)$ & $(1,1,2,2,3,5)$ & $\frac{1}{5}(1,1,2)$ & $\frac{5}{6}$ & 5 & -1 & 2 & $5\times(1,2)$ $(1,3)$    \\
\hline
13 & 1 & $(8,9)$ & $(1,1,2,2,3,7)$ & $\frac{1}{7}(1,2,3)$ & $\frac{5}{6}$ & 5 & -1 & 2 & $5\times(1,2)$ $(1,3)$    \\
\hline
14 & 2 & $(20,21)$ & $(3,3,4,5,7,17)$ & $\frac{1}{17}(3,5,7)$ & $\frac{16}{105}$ & 1 & 1 & 2 & $8\times(1,3)$ $(2,5)$ $(3,7)$    \\
\hline
15 & 2 & $(12,16)$ & $(1,2,3,5,7,8)$ & $\frac{1}{7}(1,3,1)$ & $\frac{8}{15}$ & 4 & -1 & 2 & $(1,3)$ $(2,5)$   \\
\hline
16 & 2 & $(15,16)$ & $(1,3,4,5,5,11)$ & $\frac{1}{11}(1,3,5)$ & $\frac{8}{15}$ & 3 & 0 & 2 & $(1,3)$ $4\times(2,5)$ \\
\hline
17 & 3 & $(12,24)$ & $(3,4,5,6,7,8)$ & $\frac{1}{5}(1,2,1)$ & $\frac{2}{7}$ & 2 & 1 & 2 & $4\times (1,2)$ $(3,7)$   \\
\hline
18 & 3 & $(15,26)$ & $(2,3,5,7,8,13)$ & $\frac{1}{7}(3,1,2)$ & $\frac{11}{24}$ & 2 & 0 & 2 & $4\times (1,2)$ $(1,3)$ $(1,8)$  \\
\hline
19 & 3 & $(16,24)$ & $(2,3,4,5,11,12)$ & $\frac{1}{11}(3,4,1)$ & $\frac{9}{20}$ & 3 & 0 & 3 & $8\times (1,2)$ $(1,4)$ $(1,5)$    \\
\hline
20 & 3 & $(18,23)$ & $(1,4,6,7,9,11)$ & $\frac{1}{11}(5,2,3)$ & $\frac{281}{420}$ & 3 & 0 & 3 & $2\times(1,2)$ $(1,3)$ $(1,4)$ $(2,5)$ $(3,7)$    \\
\hline
21 & 3 & $(20,21)$ & $(1,4,5,7,8,13)$ & $\frac{1}{13}(1,4,5)$ & $\frac{27}{40}$ & 3 & 0 & 2 & $3\times(1,4)$ $(2,5)$ $(3,8)$    \\
\hline
22 & 3 & $(20,21)$ & $(2,3,4,5,7,17)$ & $\frac{1}{17}(2,5,7)$ & $\frac{27}{35}$ & 3 & 0 & 2 & $6\times(1,2)$ $(1,5)$ $(2,7)$    \\
\hline
23 & 3 & $(20,28)$ & $(2,4,5,7,13,14)$ & $\frac{1}{13}(4,5,1)$ & $\frac{27}{140}$ & 1 & 1 & 2 & $10\times(1,2)$ $(1,4)$ $(2,5)$ $2\times(2,7)$    \\
\hline
24 & 3 & $(20,36)$ & $(1,4,7,10,13,18)$ & $\frac{1}{13}(1,4,5)$ & $\frac{27}{140}$ & 2 & 0 & 2 & $2\times(1,2)$ $(1,4)$ $(2,5)$ $(1,7)$    \\
\hline
25 & 4 & $(20,28)$ & $(2,3,5,7,13,14)$ & $\frac{1}{13}(3,5,1)$ & $\frac{64}{105}$ & 3 & 0 & 2 &  $2\times(1,3)$ $(2,5)$ $2\times (2,7)$    \\
\hline
26 & 5 & $(24,30)$ & $(3,4,8,10,11,13)$ & $\frac{1}{11}(5,3,2)$ & $\frac{509}{780}$ & 2 & 1 & 2 &  $4\times(1,2)$ $(1,3)$ $3\times (1,4)$ $(2,5)$ $(6,13)$   \\
\hline
27 & 5 & $(27,28)$ & $(3,4,7,8,9,19)$ & $\frac{1}{19}(3,4,7)$ & $\frac{125}{168}$ & 2 & 1 & 2 &  $4\times(1,3)$ $4\times (1,4)$ $(3,7)$ $(3,8)$   \\
\hline
28 & 5 & $(24,28)$ & $(2,6,7,8,11,13)$ & $\frac{1}{11}(3,5,1)$ & $\frac{161}{195}$ & 3 & 1 & 2 &  $14\times(1,2)$ $(1,3)$ $(2,5)$ $(3,13)$   \\
\hline
29 & 5 & $(24,36)$ & $(2,3,7,8,17,18)$ & $\frac{1}{17}(3,8,1)$ & $\frac{125}{168}$ & 4 & 0 & 2 &  $6\times(1,2)$ $(1,3)$ $(3,7)$ $(3,8)$   \\
\hline
30 & 5 & $(24,34)$ & $(1,3,8,11,13,17)$ & $\frac{1}{13}(1,3,4)$ & $\frac{125}{132}$ & 5 & -1 & 2 &  $(1,3)$ $(1,4)$ $(2,11)$   \\
\hline
31 & 7 & $(39,42)$ & $(6,9,11,13,14,21)$ & $\frac{1}{11}(5,2,3)$ & $\frac{11}{45}$ & 1 & 1 & 5 &  $2\times (1,2)$ $5\times(1,3)$ $(2,5)$   $(1,9)$\\
\hline
32 & 7 & $(30,44)$ & $(4,5,6,11,19,22)$ & $\frac{1}{19}(4,5,3)$ & $\frac{343}{660}$ & 2 & 1 & 2 &  $5\times (1,2)$ $(1,3)$ $(1,4)$ $(2,5)$ $2\times (3,11)$  \\
\hline
33 & 7 & $(40,45)$ & $(1,8,9,15,20,25)$ & $\frac{1}{25}(1,8,9)$ & $\frac{343}{360}$ & 3 & 0 & 2 & $(1,3)$ $(1,4)$ $(2,5)$ $(1,8)$ $(2,9)$ \\
\hline
34 & 9 & $(30,54)$ & $(4,6,10,13,15,27)$ & $\frac{1}{13}(5,4,3)$ & $\frac{14}{15}$ & 3 & 1 & 2 & $13\times(1,2)$ $3\times(1,3)$ $2\times (1,4)$ $2\times (2,5)$  \\
\hline
35 & 3 & $33$ & $(2,3,5,7,13)$ & $\frac{1}{13}(2,3,5)$ $\frac{1}{7}(3,1,2)$ & $\frac{7}{30}$ & 2 & 0 & 4 & $3\times(1,2)$ $(1,3)$  $2\times (1,5)$  \\
\hline
36 & 1 & $(15,22)$ & $(2,3,4,5,11,11)$ & $\frac{1}{11}(2,3,5)$ $\frac{1}{11}(2,3,5)$ & $\frac{1}{60}$ & 1 & 1 & 3 & $7\times(1,2)$  $2\times (1,3)$ $(1,4)$  $2\times (2,5)$  \\
\hline
37 & 1 & $(8,10)$ & $(1,1,2,3,5,5)$ & $\frac{1}{5}(1,1,2)$ $\frac{1}{5}(1,1,2)$ & $\frac{1}{3}$ & 4 & -1 & 3 & $2\times(1,2)$  $(1,3)$  \\
\hline
38 & 1 & $(6,8)$ & $(1,1,1,2,4,4)$ & $\frac{1}{4}(1,1,1)$ $\frac{1}{4}(1,1,1)$ & $\frac{1}{1}$ & 7 & -2 & 3 & $2\times(1,2)$  \\
\hline
39 & 3 & $(20,26)$ & $(1,4,5,7,13,13)$ & $\frac{1}{13}(1,4,5)$ $\frac{1}{13}(1,4,5)$ & $\frac{27}{70}$ & 3 & 0 & 3 & $2\times(1,4)$ $2\times (2,5)$ $(3,7)$  \\
\hline
40 & 3 & $(12,28)$ & $(2,3,5,6,7,14)$ & $\frac{1}{7}(3,1,2)$ $\frac{1}{7}(3,1,2)$ & $\frac{7}{15}$ & 3 & 0 & 3 & $6\times(1,2)$ $4\times (1,3)$ $(2,5)$ \\
\hline
41 & 2 & $(10,12,14)$ & $(2,3,4,5,6,7,7)$ & $\frac{1}{7}(1,2,3)$ $\frac{1}{7}(1,2,3)$ & $\frac{1}{3}$ & 2 & 0 & 3 & $2\times(1,2)$ $4\times (1,3)$ \\
\hline
42 & 2 & $(10,12,12)$ & $(2,3,4,5,5,6,7)$ & $\frac{1}{7}(1,2,3)$  & $\frac{13}{30}$ & 2 & 0 & 2 & $(1,2)$ $(1,3)$ $2\times (2,5)$\\
\hline
43 & 3 & $(12,14,18)$ & $(3,4,5,6,7,7,9)$ & $\frac{1}{7}(1,2,3)$ $\frac{1}{7}(1,2,3)$  & $\frac{7}{15}$ & 2 & 0 & 3 & $2\times (1,2)$ $6\times (1,3)$ $(2,5)$\\
\hline

\end{longtable}
}

{\tiny
\begin{longtable}{|l|l|l|l|l|l|l|l| p{4.5cm} |}
	\caption{Minimal $3$-folds of general type, II}\label{tableAp}\\
		\hline
		No.&$\alpha$ & deg & weight & B-weight & $\Vol$&$P_2$ & $\chi$ &basket \\ \hline	
\endfirsthead
\multicolumn{5}{l}{{ {\bf \tablename\ \thetable{}} \textrm{-- continued from previous page}}} \\
\hline 
		No.&$\alpha$ & deg & weight & B-weight & $\Vol$&$P_2$ & $\chi$ &basket \\ \hline	
\endhead

 \multicolumn{5}{l}{{\textrm{Continued on next page}}} \\ \hline
\endfoot
\hline 
\endlastfoot
\hline
1 & 2 & $(20,21)$ & $(3,4,5,6,7,14)$ & $\frac{1}{14}(3,4,5)$ & $\frac{3}{35}$ & 1 & 1 & $4\times(1,3)$ $(2,5)$ $(3,7)$\\
\hline
2 & 2 & $(15,22)$ & $(2,3,4,5,10,11)$ & $\frac{1}{10}(3,4,1)$ & $\frac{2}{15}$ & 2 & 0 & $(1,3)$ $(2,5)$ \\
\hline
3 & 2 & $(15,21)$ & $(1,3,4,5,7,14)$ & $\frac{1}{14}(3,4,5)$ & $\frac{44}{105}$ & 3 & 0 & $(1,3)$ $(2,5)$ $(3,7)$ \\
\hline
4 & 2 & $(15,20)$ & $(1,3,4,5,6,14)$ & $\frac{1}{14}(3,4,5)$ & $\frac{7}{15}$ & 3 & 0 & $3\times(1,3)$ $(2,5)$ \\
\hline
5 & 2 & $(15,17)$ & $(1,3,4,5,7,10)$ & $\frac{1}{10}(1,3,4)$ & $\frac{44}{105}$ & 3 & 0 & $(1,3)$ $(2,5)$ $(3,7)$\\
\hline
6 & 2 & $(15,16)$ & $(1,3,4,5,6,10)$ & $\frac{1}{10}(1,3,4)$ & $\frac{7}{15}$ & 3 & 0 & $3\times(1,3)$ $(2,5)$ \\
\hline
7 & 2 & $(15,17)$ & $(1,3,3,4,5,14)$ & $\frac{1}{14}(3,4,5)$ & $\frac{4}{5}$ & 4 & 0 & $6\times(1,3)$ $(2,5)$ \\
\hline
8 & 2 & $(12,15)$ & $(1,2,4,5,6,7)$ & $\frac{1}{7}(1,2,3)$ & $\frac{5}{6}$ & 3 & -1 & $(1,2)$ $(1,3)$ \\
\hline
9 & 3 & $(25,28)$ & $(4,5,6,7,9,19)$ & $\frac{1}{19}(4,5,7)$ & $\frac{17}{140}$ & 1 & 1 & $2\times(1,2)$ $(1,4)$ $(2,5)$ $(2,7)$\\
\hline
10 & 3 & $(18,28)$ & $(4,6,6,7,9,11)$ & $\frac{1}{11}(5,2,3)$ & $\frac{2}{15}$ & 2 & 1 & $7\times(1,2)$ $(1,3)$ $(2,5)$\\
\hline
11 & 3 & $(12,25)$ & $(3,4,5,6,7,9)$ & $\frac{1}{7}(1,2,3)$ & $\frac{1}{3}$ & 2 & 0 & $2\times(1,2)$ $(1,3)$ \\
\hline
12 & 3 & $(16,30)$ & $(1,3,6,8,10,15)$ & $\frac{1}{5}(2,1,1)$ & $\frac{1}{2}$ & 4 & -1 & $3\times(1,2)$ \\
\hline
13 & 3 & $(12,26)$ & $(2,3,5,6,6,13)$ & $\frac{1}{5}(1,2,1)$ & $\frac{1}{2}$ & 4 & 0 & $9\times(1,2)$ \\
\hline
14 & 3 & $(14,24)$ & $(2,3,3,7,8,12)$ & $\frac{1}{4}(1,1,1)$ & $\frac{1}{2}$ & 4 & -1 & $3\times(1,2)$ \\
\hline
15 & 3 & $(12,20)$ & $(2,3,5,6,6,7)$ & $\frac{1}{7}(3,1,2)$ & $\frac{5}{6}$ & 4 & 0 & $7\times(1,2)$ $(1,3)$ \\
\hline
16 & 3 & $(16,18)$ & $(2,3,4,6,7,9)$ & $\frac{1}{7}(1,2,3)$ & $\frac{5}{6}$ & 4 & 0 & $13\times(1,2)$ $(1,3)$ \\
\hline
17 & 3 & $(18,22)$ & $(1,3,6,7,9,11)$ & $\frac{1}{7}(1,2,3)$ & $\frac{5}{6}$ & 4 & -1 & $(1,2)$ $(1,3)$ \\
\hline
18 & 3 & $(14,20)$ & $(2,2,3,5,9,10)$ & $\frac{1}{9}(2,3,1)$ & $\frac{9}{10}$ & 4 & 0 & $15\times(1,2)$ $(1,5)$ \\
\hline
19 & 4 & $(24,27)$ & $(2,5,8,9,11,12)$ & $\frac{1}{11}(2,5,3)$ & $\frac{13}{30}$ & 2 & 0 & $(1,2)$ $2\times(2,5)$ \\
\hline
20 & 4 & $(15,24)$ & $(3,4,5,7,8,8)$ & $\frac{1}{7}(1,3,2)$ & $\frac{5}{6}$ & 4 & 0 & $(1,2)$ $(1,3)$ \\
\hline
21 & 4 & $(12,27)$ & $(2,5,6,6,7,9)$ & $\frac{1}{7}(1,3,1)$ & $\frac{8}{15}$ & 3 & 0 & $3\times(1,3)$ $(1,5)$ \\
\hline
22 & 4 & $(25,27)$ & $(1,5,6,8,9,19)$ & $\frac{1}{19}(1,5,9)$ & $\frac{44}{45}$ & 4 & 0 & $(1,3)$ $(1,5)$ $(4,9)$ \\
\hline
23 & 4 & $(24,25)$ & $(1,5,6,8,9,16)$ & $\frac{1}{16}(1,5,6)$ & $\frac{44}{45}$ & 4 & 0 & $(1,3)$ $(1,5)$ $(4,9)$ \\
\hline
24 & 6 & $(25,36)$ & $(3,4,5,9,16,18)$ & $\frac{1}{16}(3,5,2)$ & $\frac{4}{5}$ & 4 & 0 & $(2,5)$  \\
\hline
25 & 9 & $70$ & $(4,6,7,9,35)$ & $\frac{1}{7}(2,3,1)$ $\frac{1}{7}(2,3,1)$ & $\frac{11}{12}$ & 4 & 0 & $7\times(1,2)$ $2\times(1,3)$  \\
\hline
26 & 2 & $(15,28)$ & $(1,3,4,5,14,14)$ & $\frac{1}{14}(3,4,5)$ $\frac{1}{14}(3,4,5)$ & $\frac{4}{15}$ & 3 & 0 & $2\times(1,3)$ $2\times(1,4)$ $2\times (2,5)$  \\
\hline
27 & 2 & $(15,20)$ & $(1,3,4,5,10,10)$ & $\frac{1}{10}(1,3,4)$ $\frac{1}{10}(1,3,4)$ & $\frac{4}{15}$ & 3 & 0 & $2\times(1,3)$ $2\times(1,4)$ $2\times (2,5)$  \\
\hline
28 & 3 & $(18,28)$ & $(3,4,6,7,9,14)$ & $\frac{1}{7}(1,2,3)$ $\frac{1}{7}(1,2,3)$ & $\frac{1}{6}$ & 2 & 0 & $5\times(1,2)$ $2\times(1,3)$  \\
\hline
29 & 3 & $(14,16)$ & $(2,3,3,4,7,8)$ & $\frac{1}{4}(1,1,1)$ $\frac{1}{4}(1,1,1)$ & $\frac{1}{1}$ & 5 & -1 & $6\times(1,2)$  \\
\hline
30 & 2 & $(12,14,15)$ & $(2,4,5,6,7,7,8)$ & $\frac{1}{7}(1,2,3)$ $\frac{1}{7}(1,2,3)$ & $\frac{1}{6}$ & 2 & 0 & $2\times(1,2)$ $2\times (1,3)$  \\
\hline
31 & 2 & $(11,12,14)$ & $(2,4,4,5,6,7,7)$ & $\frac{1}{7}(1,2,3)$ $\frac{1}{7}(1,2,3)$ & $\frac{4}{15}$ & 3 & 0 & $2\times(1,2)$ $2\times (1,3)$  \\
\hline
32 & 2 & $(10,12,14)$ & $(2,4,4,5,5,6,7)$ & $\frac{1}{5}(1,2,1)$ $\frac{1}{5}(1,2,1)$ & $\frac1{6}$ & 3 & 0 & $2\times(1,2)$  \\
\hline
33 & 2 & $(10,12,15)$ & $(2,3,4,5,6,7,8)$ & $\frac{1}{7}(1,2,3)$  & $\frac1{3}$ & 2 & 0 & $(1,2)$ $(1,3)$  \\
\hline
34 & 2 & $(10,11,12)$ & $(2,3,4,4,5,6,7)$ & $\frac{1}{7}(1,2,3)$  & $\frac1{2}$ & 3 & 0 & $(1,2)$ $(1,3)$  \\
\hline
35 & 2 & $(9,10,12)$ & $(2,2,3,4,5,6,7)$ & $\frac{1}{7}(1,2,3)$  & $\frac{5}{6}$ & 4 & -1 & $(1,2)$ $(1,3)$  \\
\hline
36 & 3 & $(10,12,18)$ & $(3,3,4,5,6,7,9)$ & $\frac{1}{7}(1,2,3)$  & $\frac{5}{6}$ & 4 & -1 & $(1,2)$ $(1,3)$  \\
\hline

\end{longtable}
}

We illustrate on how to do the manual verification for several typical examples.

\begin{exmp}[Table~\ref{tableA}, No.~1] Consider the general weighted complete intersection 
$$X=X_{20,21}\subset \mathbb{P}(3,4,5,7,8,13),$$
which is clearly well-formed and quasi-smooth by Definition~\ref{wellform} and Proposition~\ref{2.8}. One also knows that $\alpha = 1$, $p_g = 0$, $P_2=0$, and $\chi(\mathcal{O}_X) = 1$. The set of singularities of $X$ is
	$$\text{\rm Sing}(X) = \{\frac{1}{4}(3,3,1),\ \frac{1}{4}(3,3,1),\, \frac{1}{8}(3,7,5), \ Q=\frac{1}{13}(3,4,5)\},$$
	where the first $3$ singularities are terminal, while the last one is the  unique non-canonical singular point.
	
	It's easy to determine $\text{\rm Sing}(X)$ using Proposition~\ref{non iso can}. 
	Denote $P_0, \dots, P_5$ to be the vertices. As $3 | 21$, $4 | 20$, $5|20$, $7|21$  $P_i\not \in X$ for $i\in\{0,1,2,3\}$ while $P_{4}, P_5\in X$. For $P_4$, we have $8 | (20-4)$ and $8|(21-5)$ so $P_4$ is a singularity of type $\frac{1}{8}(3,7,5)$. Similarly we can determine the type of $P_5$, which is $Q=\frac{1}{13}(3,4,5)$. As $\gcd (4,8) =4$, there are $\lfloor \frac{4\times20}{4\times8}\rfloor =2$ singular points on $P_1P_4\cap X$, each of type $\frac{1}{4}(3,3,1)$. There are no other singular points on $X$ since the other weights are pairwisely coprime.

	For applying Theorem~\ref{nefness 2}, we take 
	$$(b_1,b_2,b_3,b_4,b_5,b_6) = (3,4,5,7,8,13),$$
	$$(e_1,e_2,e_3)=(3,4,5),$$
	$r=13$, and $k=2$. Conditions $(1),(2),(4)$ follow from direct computations (or Remark~\ref{nefrk}). Condition $(3)$ means to check that a general curve $C_{20,21}\subset\mathbb{P}(4,7,8,13)$ is irreducible, which follows from the fact that this curve is actually quasi-smooth. On the other hand, we can also argue that  $\alpha e_{j} = b_{j}(r-\sum_{i=1}^3e_i)$ for all $j\in\{1,2,3\}$ and $Z_{20,21}\subset \mathbb{P}(7,8,13)$ is a finite set.
	
	So by Theorem~\ref{nefness 2}, we can take a weighted blow-up $f:\widetilde{X}\to X$ at the point $Q$ with weight $(3,4,5)$ such that $K_{\widetilde{X}}$ is nef. On the exceptional divisor $E$ there are $3$ new singularities:
	$$\frac{1}{3}(2,1,2),\ \frac{1}{4}(3,3,1),\ \frac{1}{5}(3,4,2),$$ all of which are terminal. Hence $\widetilde{X}$ is a minimal $3$-fold and $\widehat{X}=\widetilde{X}$. Applying the volume formula for weighted blow-ups (cf. Proposition~\ref{wb}), we get $\Vol(\widehat{X})=K_{\widetilde{X}}^3 = \frac{1}{120}$.
	Since $\rho(X)=1$ by \cite[Theorem 3.2.4(i)]{WPS}, after one weighted blow-up the Picard number becomes $2$.
	 Finally, we collect the singularities of $\widehat{X}$ and obtain the Reid basket 
	$$B_{\widehat{X}} = \{ (1,3),\ 3\times(1,4),\ (2,5),\ (3,8)\}.$$\qed
\end{exmp}

\begin{exmp}[Table~\ref{tableA}, No.~31] Consider the general weighted complete intersection
	$$X=X_{39,42}\subset \mathbb{P}(6,9,11,13,14,21), $$
which is well-formed and quasi-smooth by Definition~\ref{wellform} and Proposition~\ref{2.8}.
	It is clear that $\alpha = 7$, $P_2 = 1$ and $\chi(\mathcal{O}_X) = 1$. Moreover it's easy to determine 
	$$\text{\rm Sing}(X) = \{\frac{1}{2}(1,1,1),\ 4\times \frac{1}{3}(2,1,2),\ \frac{1}{7}(1,5,1), \frac{1}{9}(1,1,8),\ Q = \frac{1}{11}(5,2,3)\}.$$ Note that $Q$ is the unique non-canonical singular point, while $\frac{1}{7}(1,5,1)$ is an isolated canonical singular point.

	Here we illustrate how to use Proposition~\ref{non iso can} to determine $\text{\rm Sing}(X)$. 
	Denote $P_0, \dots, P_5$ to be the vertices. Note that  $P_1,P_2  \in X$ while $P_{0}, P_3, P_4, P_5\not\in X$. For $P_2$, we have $11 | (39-6)$, $11|(42-9)$ so $P_4$ is a singularity of type $\frac{1}{11}(13,14,21)=\frac{1}{11}(2, 3, 10)=\frac{1}{11}(5, 2, 3)$. Similarly $P_1$ is a singularity of type $\frac{1}{9}(1,1,8)$. For edges, there are singular points on $P_0P_4\cap X$ and $P_4P_5\cap X$. For $P_0P_4\cap X$, there is $\lfloor \frac{2\times 42}{6\times 14}\rfloor=1$ singular point of type $\frac{1}{2}(11,13,21)=\frac{1}{2}(1,1,1)$;  for $P_4P_5\cap X$, there is $\lfloor \frac{7\times 42}{14\times 21}\rfloor=1$ singular point of type $\frac{1}{7}(6,9,13)=\frac{1}{7}(1,5,1)$. On the only singular face $P_0P_1P_5$, as $\gcd (6,9,21)=3$, the number of singular points on this face is obtained by $N= \frac{3\times 39\times42}{6\times9\times  21} - \frac{3}{9} = 4$, each of type $\frac{1}{3}(11,13,14)=\frac{1}{3}(2,1,2)$. There are no other singular points.

	Take $ (b_1,...,b_6) = (13,14,21,6,9,11)$, $(e_1,e_2,e_3) = (5,2,3)$, 
	$r = 11$, and $k = 3$. One can check that the conditions of Theorem~\ref{nefness 2} are satisfied and hence, after a weighted blow-up with weight $(5,2,3)$ at $Q$, we get
	$f:\widetilde{X}\rightarrow X$
	so that $K_{\widetilde{X}}$ is nef and $K^3_{\widetilde{X}}=\frac{11}{45}$.	By Proposition~\ref{wb}, we see that
	$$\text{\rm Sing}(\widetilde{X}) = \{2\times \frac{1}{2}(1,1,1),\ 5\times \frac{1}{3}(2,1,2),\ \frac{1}{5}(1,3,4),\frac{1}{7}(1,5,1),\  
	\frac{1}{9}(1,1,8)\}.$$
 These are all canonical singularities, among which only one is non-terminal: $\frac{1}{7}(1,5,1)$.
	By Lemma~\ref{can sing res}, there exists a terminalization $g:\widehat{X}\rightarrow\widetilde{X}$
	such that $B_{\widehat{X}}=B_{\widetilde{X}}
	= \{2\times (1,2),\ 5\times (1,3),\ (2,5),\ (1,9)\} $
	and $\rho(\widehat{X}) = \rho(\widetilde{X})+ 3 = 5$.
Since $g$ is crepant, $K_{\widehat{X}}^3=K^3_{\widetilde{X}}=\frac{11}{45}$. 
\qed
\end{exmp}

 \begin{rem} An interesting phenomenon is that some of our examples have the same deformation invariants ( $\Vol(\widehat{X})$, $P_2$, and basket) as some known examples found by Iano-Fletcher \cite{Fle00} or Chen, Jiang and Li's previous paper\cite{CJL24}. However since birationally equivalent minimal varieties must have the same Picard number (\cite[Theorem 3.52(2)]{K-M}), so most of our examples are birationally different from their examples. Those examples sharing the same deformation invariants with previous examples may actually be mutually deformation equivalent.
\end{rem}

\subsection{Examples of minimal $3$-folds of general type near the Noether line}

 The geography problem of $3$-folds of general type is a very interesting topic in explicit birational geometry. Among which the ``Noether inequality in dimension $3$'' is one of the most challenging problems. This inequality was proved by Chen, Chen and Jiang (\cite{Noether, Noether_Add}) for the case where $p_g\leq 4$ or $p_g\geq 11$, and the remaining cases where $5\leq p_g \leq 10$ have recently been proved by Chen, Hu, and Jiang (\cite[Theorem 1.4]{CHJ24}):

\begin{thm}
    \label{Noether} Any minimal projective $3$-fold $X$ of general type satisfies the inequality
$$K_X^3\geq \frac{4}{3}p_g(X)-\frac{10}{3}.$$
\end{thm}

We say $K^3= \frac{4}{3}p_g-\frac{10}{3}$ is the first Noether line. In fact, Hu and Zhang (\cite{HZ25}) proved that there exists two other Noether lines and they have the following interesting characterizations.

\begin{thm}\label{Noether lines}\cite[Theorem 1.5]{HZ25}
    Let $X$ be a minimal $3$-fold of general type with $p_g(X)\geq 11$.
    \begin{itemize}
        \item If $X$ is on the first Noether line $K^3= \frac{4}{3}p_g-\frac{10}{3}$, then $p_g(X)\equiv 1 \mod 3$.
        \item If $X$ is strictly above the first Noether line, we have the inequality: $K^3\geq \frac{4}{3}p_g-\frac{19}{6}$, and if $X$ is on the second Noether line $K^3 =\frac{4}{3}p_g-\frac{19}{6}$, then $p_g(X)\equiv 2 \mod 3$.
        \item If $X$ is strictly above the second Noether line, we have the inequality: $K^3\geq \frac{4}{3}p_g-3$, and if $X$ is on the third Noether line $K^3 =\frac{4}{3}p_g-3$, then $p_g(X)\equiv 0 \mod 3$.
    \end{itemize}
\end{thm}

The effectivity of Theorem~\ref{nefness 2} and Theorem~\ref{nefness 3} make it possible for us to search those concrete $3$-folds near the three Noether lines. We provide here several new examples in Tables~\ref{tableC} and~\ref{tableC+}. It's interesting to note that although most of the examples below are birationally different than those constructed by Chen, Jiang and Li \cite{CJL24} (since they have different Picard numbers), many of them share the same deformation invariants with some of their previous examples, and those exmaples may indeed be deformation equivalent.

The description of contents of Tables~\ref{tableC} is similar to that of Table~\ref{tableA}, except for some examples with non-isolated canonical singularities, the computation of the Picard number is simply omitted. Here the last column is the distance $\Delta$ to the first Noether line, that is, $\Delta=\Vol(\widehat{X})-\frac{4}{3}p_g(\widehat{X})+\frac{10}{3}$.

 {\tiny
\begin{longtable}{|l|l|l|l|l|l|l|l|l| l |l|}
 			\caption{ Minimal $3$-folds of general type near the Noether line, I}\label{tableC}\\
		\hline
			No.&$\alpha$ & deg & weight & B-weight & $\Vol$&$P_2$ & $p_g$&$\rho$ &basket & $\Delta$ \\ \hline	
\endfirsthead
\multicolumn{5}{l}{{ {\bf \tablename\ \thetable{}} \textrm{-- continued from previous page}}} \\
\hline 
			No.&$\alpha$ & deg & weight & B-weight & $\Vol$&$P_2$ & $p_g$&$\rho$ &basket & $\Delta$ \\ \hline	
\endhead

 \multicolumn{4}{l}{{\textrm{Continued on next page}}} \\ \hline
\endfoot
\hline 
\endlastfoot
\hline
1 & 2 & $(8,12)$ & $(1,1,2,3,5,6)$ & $\frac{1}{5}(1,1,1)$ & $\frac{8}{3}$ & $11$ & $4$ & $2$ & $2\times (1,3)$ & $\frac{2}{3}$ \\
\hline
2 & 3 & $(12,28)$ & $(1,1,4,6,11,14)$ & $\frac{1}{11}(1,4,3)$ & $\frac{9}{4}$ & $11$ & $4$ & $2$ & $2\times (1,2)$ $(1,3)$ $(1,4)$ & $\frac{1}{4}$ \\
\hline
3 & 3 & $(12,22)$ & $(1,1,4,6,8,11)$ & $\frac{1}{8}(1,1,3)$ & $\frac{9}{4}$ & $11$ & $4$ & $2$ & $2\times (1,2)$ $(1,3)$ $(1,4)$ & $\frac{1}{4}$ \\
\hline
4 & 3 & $(13,18)$ & $(1,1,4,6,7,9)$ & $\frac{1}{7}(1,1,2)$ & $\frac{9}{4}$ & $11$ & $4$ & $2$ & $2\times (1,2)$ $(1,3)$ $(1,4)$ & $\frac{1}{4}$ \\
\hline
5 & 2 & $(8,10)$ & $(1,1,1,3,5,5)$ & $\frac{1}{5}(1,1,1)$ $\frac{1}{5}(1,1,1)$ & $\frac{16}{3}$ & $18$ & $6$ & $2$ & $(1,3)$ & $\frac{2}{3}$ \\
\hline
6 & 3 & $21$ & $(1,1,3,5,8)$ & $\frac{1}{5}(2,1,1)$ $\frac{1}{8}(1,1,3)$ & $\frac{7}{2}$ & $14$ & $5$ & $4$ & $(1,2)$  & $\frac{1}{6}$\\
\hline
7 & 7 & $45$ & $(1,2,7,11,17)$ & $\frac{1}{11}(4,3,1)$ $\frac{1}{17}(1,2,7)$ & $\frac{17}{4}$ & $15$ & $5$ & $10$ & $2\times(1,2)$ $(1,4)$ & $\frac{11}{12}$\\
\hline
8 & 5 & $32$ & $(1,1,5,8,12)$ & $\frac{1}{4}(1,1,1)$ $\frac{1}{12}(1,1,5)$ & $6$ & $21$ & $7$ & $7$ &  & $0$\\
\hline
9 & 5 & $(20,24)$ & $(1,1,5,8,12,12)$ & $\frac{1}{4}(1,1,1)$ $\frac{1}{12}(1,1,5)$ $\frac{1}{12}(1,1,5)$ & $6$ & $21$ & $7$ & $8$ &  & $0$\\
\hline
10 & $2$ & $(9,12)$ & $(1,1,2,4,5,6)$ & $\frac{1}{5}(1,1,1)$ & $2$ & $10$ & $4$ & &  & $0$ \\
\hline
11 & $2$ & $(7,12)$ & $(1,1,2,2,5,6)$ & $\frac{1}{5}(1,1,1)$ & $4$ & $14$ & $5$ & &  & $\frac{2}{3}$ \\
\hline
12 & $6$ & $37$ & $(1,1,6,9,14)$ & $\frac{1}{9}(2,3,1)$ $\frac{1}{14}(1,1,6)$ & $\frac{15}{2}$ & $25$ & $8$ & &  $(1,2)$ & $\frac{1}{6}$ \\
\hline

\end{longtable}
}

We also use Theorem \ref{nefness 4} to find several interesting examples near the Noether lines. The following are $4$ new examples obtained in this way, which has the same deformation invariants as some of the examples found by Chen, Jiang and Li.

 {\tiny
\begin{longtable}{|l|l|l|l|l|l|l|l|l| l |l|}
 			\caption{ Minimal $3$-folds of general type near the Noether line, II}\label{tableC+}\\
		\hline
		No.&$\alpha$ & deg& weight & B-weight & $\Vol$&$P_2$ & $p_g$ & $\rho$ &basket & $\Delta$\\ \hline
\endfirsthead
\multicolumn{10}{l}{{ {\bf \tablename\ \thetable{}} \textrm{-- continued from previous page}}} \\
\hline 
		No.&$\alpha$ & deg& weight & B-weight & $\Vol$&$P_2$ & $p_g$ &basket & $\Delta$ \\ \hline
\endhead

 \multicolumn{10}{l}{{\textrm{Continued on next page}}} \\ \hline
\endfoot
\hline 
\endlastfoot
\hline
1 & $3$ & $25$ & $(1,1,3,5,12)$ & $\frac{1}{12}(1,3,5)$ / $\frac{1}{5}(2,1,1)$ & $\frac{7}{2}$ & $14$ & $5$ & &  $(1,2)$ & $\frac{1}{6}$ \\
\hline
2 & $4$ & $35$ & $(1,1,5,7,17)$ & $\frac{1}{17}(1,5,7)$ / $\frac{1}{7}(2,3,1)$ & $\frac{109}{30}$ & $15$ & $5$ & 3 &  $(1,2)$ $(1,3)$ $(2,5)$& $\frac{3}{10}$ \\
\hline
3 & $5$ & $(16,30)$ & $(1,1,5,8,11,15)$ & $\frac{1}{11}(1,1,4)$ / $\frac{1}{4}(1,1,1)$ & $6$ & $21$ & $7$ & 7 &  & $0$ \\
\hline
4 & $6$ & $(20,36)$ & $(1,1,7,10,13,18)$ & $\frac{1}{13}(1,1,5)$ / $\frac{1}{5}(1,1,2)$ & $\frac{85}{14}$ & $22$ & $7$ & 3 &  $(1,2)$  $(2,7)$& $\frac{1}{14}$ \\
\hline

\end{longtable}
}

\begin{rem}\label{remark noether line}
\begin{enumerate}
\item The minimal $3$-folds $\widehat{X}_{32}$ $\widehat{X}_{20,24}$ and $\widehat{X}_{16,30}$ corresponding to Table~\ref{tableC}, No.~8  No.~9 and Table~\ref{tableC+}, No.~3 are new examples on the first Noether line. While the minimal $3$-folds $\widehat{X}_{21}$ $\widehat{X}_{37}$ and $\widehat{X}_{25}$ corresponding to Table~\ref{tableC}, No.~6 and No.~12 and Table~\ref{tableC+}, No.~1 are new examples on the second Noether line. In particular, $\widehat{X}_{37}$ is a completely new example on the second Noether line for $p_g =8$. Recently Hu and Zhang \cite{HZ25} use another construction to obtain various minimal $3$-folds on the three Noether lines, where their construction requires $p_g(X)\geq 11$.

\item During the search we also found an example that could be found using only Chen, Jiang and Li's original methods. That is, the weighted blow up $\widehat{X}_{106}$ of $X_{106}\subset \mathbb{P}(1,1,15,21,53)$ with B-weight $\frac{1}{21}(2,9,1)$. This example has $p_g=17$ and $\mathrm{Vol}(\widehat{X}_{106})=\frac{39}{2}$, which is a new example precisely on the second Noether line.

\item Our examples, together with Chen, Jiang and Li's previous examples, show that the assumption $p_g(X)\geq 11$ in Theorem \ref{Noether lines} is indeed optimal, as we have examples with canonical volume between two Noether lines for $p_g=5,7,8,10$. See for example our Table~\ref{tableC+}, No.~2 and No.~4 and \cite[Table 10, No.3, No.8, No.9 and No.10]{CJL24}.

\end{enumerate}
\end{rem}

All items in Table~\ref{tableC} and Table~\ref{tableC+} can be manually verified similar to previous cases.
We illustrate the explicit computations for two very interesting examples.

\begin{exmp}[Table~\ref{tableC}, No.~9]\label{ex cj3} Consider the general complete intersection
$$X=X_{20,24}\subset \mathbb{P}(1,1,5,8,12,12)$$
which is verified to be well-formed and quasi-smooth. It is also clear that $\alpha = 5$, $p_g = 7$, $P_2=21$. The set of singularities of $X$ is
	$$\text{\rm Sing}(X) = \{Q_1=\frac{1}{4}(1,1,1),\ Q_2=\frac{1}{12}(1,1,5),\ Q_3=\frac{1}{12}(1,1,5)\},$$
	where $Q_1$ $Q_2$ and $Q_3$ are all non-canonical, but they are all contained in a single 2-dimensional stratum $\Pi = (x_0=x_1=x_2=0)$.
	
	We are in the situation of \textbf{Case 2} of Theorem~\ref{nefness 3}. To check the conditions, we take 
	$$(b_1,b_2,b_3,b_4,b_5,b_6) = (1,1,5,8,12,12),$$
	$$(e_{1,1},e_{2,1},e_{3,1})=(1,1,1),$$ $$(e_{1,2},e_{2,2},e_{3,2})=(1,1,5),$$ $$(e_{1,3},e_{2,3},e_{3,3})=(1,1,5),$$
	$r_1=4,r_2=r_3=12$, and $k=3$. Conditions $(1),(2),(4)$ follow from direct computations. Condition $(3)$ means that the general curve $C_{20,24}\subset\mathbb{P}(5,8,12,12)$ should be irreducible, which is true since this curve is actually quasi-smooth.
	
	So by Theorem~\ref{nefness 3}, we can take a weighted blow-up $f:\widetilde{X}\to X$ at those points $Q_i$ with corresponding weights such that $K_{\widetilde{X}}$ is nef. On the exceptional divisors $E_i$ there are $2$ new singular points of type $\frac{1}{5}(1,1,3),$ which is canonical but not terminal.
	Hence by Lemma~\ref{can sing res}, there exists a terminalization $g:\widehat{X}\rightarrow\widetilde{X}$, where $\widehat{X}$ is a minimal $3$-fold. Applying the volume formula for weighted blow-ups (cf. Proposition~\ref{wb}), we get 
	$K_{\widehat{X}}^3=K_{\widetilde{X}}^3 = 6$.
	Since $\rho(X)=1$ by \cite[Theorem 3.2.4(i)]{WPS}, $\rho(\widehat{X})=\rho(\widetilde{X})+2\times 2 =1 + 3 +4 =8$. 
	 Finally, we can easily see that $\widehat{X}$ is actually non-singular and its Reid basket is empty.
One interesting point of this example is that this example is a new example on the first Noether line for $p_g=7$, with the Picard number equals $8$, which is larger than previous constructed examples.
\end{exmp}

\begin{exmp}[Table~\ref{tableC}, No.~12] Consider the general hypersurface
	$$X=X_{37}\subset \mathbb{P}(1,1,6,9,14), $$
which is well-formed and quasi-smooth by Definition~\ref{wellform} and Proposition~\ref{2.8}.
	Denote $P_0, \dots, P_4$ to be the vertices. It is clear that $\alpha = 6$, $P_2 = 25$ and $p_g(X) = 8$. Moreover $X$ has two non-canonical singular points $Q_1 = P_3= \frac{1}{9}(1,6,14)=\frac{1}{9}(2,3,1)$ and $Q_2=P_4 = \frac{1}{14}(1,1,6)$. It also has non-isolated canonical singularities along the line $P_2P_4$ and a canonical singular point at $P_2 = \frac{1}{6}(1,3,2)$.

	Take $ (b_1,...,b_5) = (1,6,14,1,9)$, $(e_1,e_2,e_3) = (2,3,1)$, $(f_1,f_2,f_3) = (1,6,1)$ 
	$r_1 = 9$, and $r_2 = 14$. We are in the \textbf{Case 1} of Theorem~\ref{nefness 3}. Let $\Pi_1 =(x_1=x_2=x_3=0)$ and $\Pi_2 =(x_1=x_2=x_4=0)$. It's clear that $Q_1\in \Pi_1 $ and $Q_2\in \Pi_2 $. One can check that the conditions of Theorem~\ref{nefness 3} are satisfied and hence, after a weighted blow-up at $Q_1$ and $Q_2$ with corresponding weights, we get
	$f:\widetilde{X}\rightarrow X$
	so that $K_{\widetilde{X}}$ is nef and $K^3_{\widetilde{X}}=\frac{15}{2}$.	On the exceptional divisors, $3$ new singular points appear, with one $\frac{1}{2}(1,1,1)$ is terminal, and the other two $\frac{1}{3}(1,0,2)$, $\frac{1}{6}(1,4,1)$ are canonical but not terminal. We can still take a terminalization
    $g:\widehat{X}\rightarrow\widetilde{X}$
	such that $$B_{\widehat{X}}=B_{\widetilde{X}}
	= \{(1,2)\} $$. However since $\widetilde{X}$ has non-isolated singularities, we omitted the computation of $\rho (\widehat{X})$. This example is interesting because it is a completely new example on the second Noether line with $p_g(\widehat{X})=8$, thus attaining minimal volume.

\qed
\end{exmp}

\subsection{Examples of minimal $3$-folds of Kodaira dimension $2$}\

Another interesting consequence of Theorem~\ref{nefness 2} is that we can construct infinite series of minimal $3$-folds of Kodaira dimension $2$. We mainly focus on the general construction obtained by blowing up a single point in a codimension $2$ weighted complete intersection, but we are aware that such constructions will also appear in the more general settings using Theorem~\ref{nefness 3}.

\begin{thm}\label{construction kod 2}
Suppose $X_{2r+4,2r+2}\subset \mathbb{P}(a, b, 4, r, r+1, r+2)$ is a $3$-dimensional well-formed quasi-smooth general complete intersection of bi-degrees $(2r+4,2r+2)$ with $\alpha=r-a-b-1>0$ where $a\leq b$ and $\gcd(a,b)=1$, but $a,b,4$ are not necessarily ordered by size. Denote by $x_1$, $x_2$, $x_3$, $x_4$, $x_5$,$x_6$ the homogenous coordinates of $\mathbb{P}(a, b, 4, r, r+1, r+2)$. Suppose that $X$ has a cyclic quotient singularity at the point $Q=(x_1=x_2=x_{5}=0)$ of type $\frac{1}{r}(a,b,1)$ which is non-canonical by assumption, and $x_1$, $x_2$, $x_3$ are the local coordinates of $Q$ corresponding to the weights $\frac{a}{r}$, $\frac{b}{r}$, $\frac{1}{r}$ respectively. Let $\pi: Y\to X$ be the weighted blow-up at $Q$ with weight $(a,b,1)$.
Then $K_Y$ is nef and $\nu(Y)=2$. 
\end{thm} 

\begin{proof}
    This is a direct check of the conditions in Theorem \ref{nefness 2}. In this special situation, condition (1) and (4) are automatically satisfied (See Remark~\ref{nefrk}). For condition (3), it suffice to check a general weighted complete intersection curve $C_{2r+4,2r+2}\subset \bP(4,r,r+1,r+2) $ is irreducible, which is true since from the expression we can see $C$ is actually quasi-smooth. It remains to check condition (2), which is equivalent to compute the canonical volume, by Proposition~\ref{wb},
$$
K_Y^3=\frac{(2r+4)(2r+2)\alpha^3}{4abr(r+1)(r+2)}-\frac{(r-a-b-1)^3}{rab}=0.
$$
Hence $K_Y$ is nef and $\nu(Y)= 2$.
\end{proof}

Note that in order to use the above theorem to find concrete examples of $3$-folds of Kodaria dimension $2$, we need to assume that the weighted complete intersection in this expression is well-formed and quasi-smooth and has only canonical singularity other than the point $\frac{1}{r}(a,b,1)$. This requires some additional restrictions on $a,b$ and $r$. In fact, we have the following tables of Kodaria dimension $2$ $3$-folds.

The description of the contents of the following tables is as follows. Each row contains a well-formed quasismooth weighted complete intersection 
	 $X=X_{2r+2,2r+4}\subset \mathbb{P}(a,b,4,r,r+1,r+2).$
	 The columns of each table contain the following information: 
\begin{center}
	\begin{tabular}{r p{10cm}}
		\hline
		$\alpha$:& The amplitude of $X$, i.e., $r-a-b-1$;\\
		$\deg$:& The degrees of $X$, which is $(2r+2,2r+4)$;\\
		weight:& $(a,b,4,r,r+1,r+2)$;\\
		B-weight:& $\frac{1}{r}(a, b, 1)$, the unique non-canonical singularity to be blown up by applying Theorem~\ref{nefness 2};\\
		conditions: & Restrictions on $r$, where $\modr(r,k)$ means the smallest non-negative residue of $r$ modulo $k$.\\		
		\hline
	\end{tabular}
\end{center}

{\tiny
\begin{longtable}{|l|l|l|l|l|p{5cm}|}
\caption{Type $X_{2r+4,2r+2}\subset \bP(a,b,4,r,r+1,r+2)$}\label{tab kod 2}\\
		\hline
		No.&$\alpha$ & deg & weight & B-weight & conditions \\ \hline
\endfirsthead
\multicolumn{5}{l}{{ {\bf \tablename\ \thetable{}} \textrm{-- continued from previous page}}} \\
\hline 
No.&$\alpha$ & deg & weight & B-weight & conditions \\ \hline
\endhead

 \multicolumn{4}{l}{{\textrm{Continued on next page}}} \\ \hline
\endfoot
\hline 
\endlastfoot
\hline
1 & $r-3$ & $(2r+2,2r+4)$ & $(1,1,4,r,r+1,r+2)$ & $\frac{1}{r}(1,1,1)$ & $ r >3 $ $\mathrm{mod}(r,4)\in \{1,2\} $\\
\hline
2 & $r-4$ & $(2r+2,2r+4)$ & $(1,2,4,r,r+1,r+2)$ & $\frac{1}{r}(1,2,1)$ & $ r >4 $ $\mathrm{mod}(r,4)\in \{1,3\} $\\
\hline
3 & $r-5$ & $(2r+2,2r+4)$ & $(1,3,4,r,r+1,r+2)$ & $\frac{1}{r}(1,3,1)$ & $ r >5 $ $\mathrm{mod}(r,4)\neq 3 $ $\mathrm{mod}(r,3)\neq 0$\\
\hline
4 & $r-6$ & $(2r+2,2r+4)$ & $(1,4,4,r,r+1,r+2)$ & $\frac{1}{r}(1,4,1)$ & $ r >6 $ $\mathrm{mod}(r,4) =1 $ \\
\hline
5 & $r-7$ & $(2r+2,2r+4)$ & $(1,4,5,r,r+1,r+2)$ & $\frac{1}{r}(1,5,1)$ & $ r >7 $ $\mathrm{mod}(r,4) \in \{1,2\} $ $\mathrm{mod}(r,5) \in \{1,2\} $ \\
\hline
6 & $r-8$ & $(2r+2,2r+4)$ & $(1,4,6,r,r+1,r+2)$ & $\frac{1}{r}(1,6,1)$ & $ r >8 $ $\mathrm{mod}(r,4) \in \{1,3\} $ $\mathrm{mod}(r,6) = 1 $ \\
\hline
7 & $r-6$ & $(2r+2,2r+4)$ & $(2,3,4,r,r+1,r+2)$ & $\frac{1}{r}(2,3,1)$ & $ r >6 $ $\mathrm{mod}(r,4) \in \{1,3\} $ $\mathrm{mod}(r,3) \neq 0  $ \\
\hline
8 & $r-8$ & $(2r+2,2r+4)$ & $(2,4,5,r,r+1,r+2)$ & $\frac{1}{r}(2,5,1)$ & $ r >8 $ $\mathrm{mod}(r,4) \in \{1,3\} $ $\mathrm{mod}(r,5) \in \{1,3\}  $ \\
\hline
9 & $r-8$ & $(2r+2,2r+4)$ & $(3,4,4,r,r+1,r+2)$ & $\frac{1}{r}(3,4,1)$ & $ r >8 $ $\mathrm{mod}(r,4) =1 $ $\mathrm{mod}(r,3) \neq 0  $ \\
\hline
10 & $r-9$ & $(2r+2,2r+4)$ & $(3,4,5,r,r+1,r+2)$ & $\frac{1}{r}(3,5,1)$ & $ r >9 $ $\mathrm{mod}(r,4)\neq 3  $ $\mathrm{mod}(r,3) \neq 0  $ $\mathrm{mod}(r,5) = 1  $  \\
\hline
11 & $r-11$ & $(2r+2,2r+4)$ & $(3,4,7,r,r+1,r+2)$ & $\frac{1}{r}(3,7,1)$ & $ r >11 $ $\mathrm{mod}(r,4) \in \{0,1\}  $ $\mathrm{mod}(r,3) \neq 0  $ $\mathrm{mod}(r,7) = 4  $  \\
\hline
12 & $r-10$ & $(2r+2,2r+4)$ & $(4,4,5,r,r+1,r+2)$ & $\frac{1}{r}(4,5,1)$ & $ r >10 $ $\mathrm{mod}(r,4) =1  $ $\mathrm{mod}(r,5) = 1  $   \\
\hline
13 & $r-12$ & $(2r+2,2r+4)$ & $(4,5,6,r,r+1,r+2)$ & $\frac{1}{r}(5,6,1)$ & $ r >12 $ $\mathrm{mod}(r,4) \in \{1,3\}  $  $\mathrm{mod}(r,5)\in \{1,2\}$ $\mathrm{mod}(r,6) = 1  $   \\
\hline
\end{longtable}
}

\begin{exmp}
We illustrate how to manually check the quasi-smoothness and singularities of $X$ if we put certain restrictions to $r$. Take [Table~\ref{tab kod 2}, No.3] for example. If $a=1,b=3$, $\alpha = r-5>0$. Since we need this complete intersection to be quasi-smooth, $\mathrm{mod}(r,3)\neq 0$, otherwise take $I=\{3,r\}$ in Theorem~\ref{2.8} will yield a contradiction. Similarly, if $\mathrm{mod}(r,4)= 3$, assume $r=4k+3$, then take $I=\{4,r+1=4k+4\}$ will also yield a contradiction. So $\mathrm{mod}(r,3)\neq 0$ and $\mathrm{mod}(r,4)\neq 3$ are the necessary conditions for $X$ being quasi-smooth. It's an easy exercise to check that these conditions are also sufficient. Denote $P_0,P_1,\dots,P_5$ to be the vertices. We also note that under these assumptions, the singularities of $X$ only happen on the vertex $P_3$ of the form $\frac{1}{r}(1,3,1)$ and on the edge $P_1P_5$ of type $(1,3)$ if $\mathrm{mod}(r,3)= 1$ or on the edge $P_1P_4$ of type $(1,3)$ if $\mathrm{mod}(r,3)= 2$. If $\mathrm{mod}(r,4)= 0$, there are also singularities on the edge $P_2P_3$ of type $(1,4)$ and if $\mathrm{mod}(r,4)= 2$ there are also singularities on the edge $P_2P_5$ of type $(1,4)$. All those singular points except $Q=P_3$ are terminal. So we can apply Theorem~\ref{construction kod 2} to obtain the corresponding minimal $3$-fold of Kodaria dimension $2$.
\end{exmp}

\begin{rem}

\begin{itemize}
    \item It's natural to ask wether there are more examples of Kodaria dimension $2$ $3$-folds of codimension $2$ and have only one non-canonical point with expression other than the one given here. This is certainly possible, but with the help of a computer program, we only discover examples in this particular form.
    \item There are also some other examples coming from the more general constructions. For example, the minimal model of $X_{28}\subset \bP(1,2,3,7,14)$ obtained by weighted blow up $2$-points of type $\frac{1}{7}(1,2,3)$ also has $\nu(Y)=2$. This can be generalized to the form: $X_{4r}\subset \bP(a,b,c,r,2r)$, with $\alpha = r-a-b-c>0$ and has only two non-canonical singular points of type $\frac{1}{r}(a,b,c)$, with additional assumptions as well.
    \item There are also examples coming from Theorem~\ref{nefness 4}. For example $X_{16,21}\subset \bP(2,3,3,7,8,13)$ is quasi-smooth and has only one non-canonical singular point $\frac{1}{13}(2,3,7)$, which after one weighted blow up still has another non canonical singular point of type $\frac{1}{7}(1,2,3)$. It's easy to check that the conditions of Theorem~\ref{nefness 4} are satisfied, and the final output is a minimal variety of Kodaria dimension $2$.
\end{itemize}    
\end{rem}

\section*{\bf Acknowledgments}

The author expresses his gratitude to his advisor Professor Meng Chen for his great support and encouragement. The author would also like to thank Sicheng Ding, Zhengjie Yu, Peien Du and Hanye Gu for useful discussions.


\begin{thebibliography}{OOOO99}
\bibitem[And18]{And18} M.~Andreatta, {\em 
Lifting weighted blow-ups}, Rev. Mat. Iberoam. {\bf 34} (2018), no. 4, 1809--1820. %

\bibitem[BCHM10]{BCHM} C.~Birkar, P.~Cascini, C.D.~Hacon, J.~McKernan, {\em Existence of minimal models for varieties of log general type}, J. Amer. Math. Soc. {\bf 23} (2010), no. 2, 405--468.

\bibitem[BK16]{B-K} G.~Brown, A.~Kasprzyk, {\em Four-dimensional projective orbifold hypersurfaces}, Exp. Math. {\bf 25} (2016), no. 2, 176--193.

\bibitem[GRD]{GRD} G.~Brown, A.~Kasprzyk, et al. {\em Graded Ring Database}, website: http://grdb.co.uk/.



\bibitem[BKR12]{BKR} G.~Brown, M.~Kerber,~M. Reid, {\em Fano $3$-folds in codimension 4, Tom and Jerry. Part I}, Compos. Math. {\bf 148} (2012), no. 4, 1171--1194. %

\bibitem[BR13]{BR1} G.~Brown, M.~Reid, {\em Diptych varieties, I}, Proc. Lond. Math. Soc. (3) {\bf 107} (2013), no. 6, 1353--1394. %

\bibitem[BR17]{BR2} G.~Brown, M.~Reid, {\em Diptych varieties. II: Apolar varieties}, In: Higher dimensional algebraic geometry--in honour of Professor Yujiro Kawamata's sixtieth birthday, pp. 41--72, Adv. Stud. Pure Math., 74, Math. Soc. Japan, Tokyo, 2017. %



\bibitem[CCJ20]{Noether} J.A.~Chen, M.~Chen, C.~Jiang, {\em The Noether inequality for algebraic $3$-folds}, Duke Math. J. {\bf 169} (2020), no. 9, 1603--1645.


\bibitem[CCJa]{Noether_Add} J.A.~Chen, M.~Chen, C.~Jiang, {\em Addendum to ``The Noether inequality for algebraic $3$-folds''}, Duke Math. J. {\bf 169} (2020), no. 11, 2199--2204. %

\bibitem[Che07]{Chen07} M.~Chen, {\em A sharp lower bound for the canonical volume of $3$-folds of general type}, Math. Ann. {\bf 337} (2007), no. 4, 887--908. %


\bibitem[CHJ24]{CHJ24}
M. Chen, Y. Hu, C. Jiang, \emph{On moduli spaces of canonical threefolds with small genera and minimal volumes}, 
arXiv:2407.01276.



\bibitem[CJL24]{CJL24}
M.~Chen, C.~Jiang, B.~Li, \emph{On minimal varieties growing from quasi-smooth weighted hypersurfaces}, J. Differential Geom. \textbf{127} (2024), no. 1, 35-76.


\bibitem[CH17]{C-H} Y.~Chen, Y.~Hu, {\em On canonically polarized Gorenstein $3$-folds satisfying the Noether equality}, Math. Res. Lett. {\bf 24} (2017), no. 2, 271--297. %

\bibitem[CR02]{CR02} A.~Corti, M.~Reid, {\em Weighted Grassmannians}, In: Algebraic geometry, pp. 141--163, de Gruyter, Berlin, 2002. %

\bibitem[Dol82]{WPS} I.~Dolgachev, {\em Weighted projective varieties}, In: Group actions and vector fields (Vancouver, B.C., 1981), pp. 34--71, Lecture Notes in Math., 956, Springer, Berlin, 1982. 

\bibitem[Eis95]{Eis95}
D. Eisenbud, \emph{Commutative Algebra with a View toward Algebraic Geometry},
Graduate Texts in Mathematics \textbf{150},
Springer-Verlag, Berlin, New York, 1995.

\bibitem[ETW21]{ETW} L.~Esser, B.~Totaro, C.~Wang, {\it Varieties of general type with doubly exponential asymptotics}, arXiv:2109.13383v2.



\bibitem[Fuj74]{Fujiki} A.~Fujiki, {\em On resolutions of cyclic quotient singularities}, Publ. Res. Inst. Math. Sci. {\bf 10} (1974/75), no. 1, 293--328. 


\bibitem[HZ25]{HZ25}
Y. Hu, T. Zhang, \emph{Algebraic threefolds of general type with small volume}, 
Math. Ann. \textbf{391} (2025), no. 1, 567-612.

\bibitem[IF00]{Fle00} A.R.~Iano-Fletcher, {\em Working with weighted complete intersections}, In: Explicit birational geometry of $3$-folds, pp. 101--173,
London Math. Soc. Lecture Note Ser., 281, Cambridge Univ. Press, Cambridge, 2000.


\bibitem[Ish18]{Ishii} S.~Ishii, {\em Introduction to singularities}, 
Second edition, Springer, Tokyo, 2018, x+236 pp.

\bibitem[JK01]{JK01} J.M.~Johnson, J.~Koll\'ar, {\em Fano hypersurfaces in weighted projective $4$-spaces}, Experiment. Math. {\bf 10} (2001), no. 1, 151--158. %


\bibitem[Kaw85]{K5} Y.~Kawamata, {\em Minimal models and the Kodaira dimension of algebraic fiber spaces}, J. Reine Angew. Math. {\bf 363} (1985), 1--46.


\bibitem[Kaw92]{K6} Y.~Kawamata, {\em Abundance theorem for minimal threefolds}, Invent. Math. {\bf 108} (1992), no. 2, 229--246. %



\bibitem[KMM85]{KMM} Y.~Kawamata, K.~Matsuda, K.~Matsuki, {\em Introduction to the minimal model problem}, 
In: Algebraic geometry, Sendai, 1985, pp. 283--360, Adv. Stud. Pure Math., 10, North-Holland, Amsterdam, 1987.

\bibitem[Kob92]{Kob} M.~Kobayashi, {\em On Noether's inequality for threefolds}, J. Math. Soc. Japan {\bf 44} (1992), no. 1, 145--156.


\bibitem[KM98]{K-M} J.~Koll{\'a}r, S.~Mori, {\em Birational geometry of algebraic varieties}, Cambridge Tracts 
in Mathematics {\bf 134}, Cambridge Univ. Press, 1998.

\bibitem[Miy88]{Mi4} Y.~Miyaoka, {\it Abundance conjecture for $3$-folds: case $\nu = 1$}, Compos. Math. {\bf 68} (1988), no. 2, 203--220. %

\bibitem[PST17]{PST17}
M. Pizzato, T. Sano, L. Tasin, \emph{Effective nonvanishing for Fano weighted complete intersections}, 
Algebra Number Theory \textbf{11} (2017), no. 10, 2369-2395.

\bibitem[Rei79]{Reid79} M.~Reid, {\em Canonical $3$-folds}, In: Journ{\'e}es de G{\'e}om{\'e}trie Alg{\'e}brique d'Angers,
juillet 1979/Algebraic Geometry, Angers, 1979, pp. 273--310, Sijthoff \& Noordhoff, Alphen aan den Rijn-Germantown, Md., 1980. 

\bibitem[Rei87]{Rei87} M.~Reid, {\em Young person's guide to canonical singularities}, In: Algebraic geometry, Bowdoin, 1985 (Brunswick, Maine, 1985), pp. 345--414,
Proc. Sympos. Pure Math., 46, Part 1, Amer. Math. Soc., Providence, RI, 1987.

\bibitem[Rei05]{Rei05} M.~Reid, {\em Constructing algebraic varieties via commutative algebra}, In: European Congress of Mathematics, pp. 655--667, Eur. Math. Soc., Z\"urich, 2005. %


\bibitem[TW21]{TW} B.~Totaro, C.~Wang, {\it Klt varieties of general type with small volume}, arXiv:2104.12200v4.



\end{thebibliography}
\end{document}